\documentclass[a4papaer, 11pt]{article}
\usepackage{amsmath,amsfonts}
\usepackage{amsmath}
\DeclareMathOperator*{\argmax}{arg\,max}

\makeatletter
\def\maketag@@@#1{\hbox{\m@th\normalfont\normalsize#1}}
\makeatother
\usepackage{amsthm}
\usepackage{hyperref} 
\usepackage{extarrows}
\usepackage{array}
\usepackage{arydshln}
\usepackage{epsfig}
\usepackage{epstopdf}
\usepackage{fancybox,amssymb}
\usepackage{color,soul}
\usepackage{graphicx}
\usepackage{epstopdf}
\usepackage{curves}
\usepackage{subfigure}
\usepackage{epic,eepic}
\usepackage{algorithm}
\usepackage{algorithmic}
\usepackage{tabularx}
\usepackage{bbm}
\usepackage{comment}
\usepackage{geometry}
\usepackage{caption}
\newtheorem{thm}{Theorem}

\newtheorem{lem}{Lemma}

\newtheorem{exam}{Example}
\newtheorem{asmp}{Assumption}

\def\rm{\mathrm}

\usepackage{natbib}
\usepackage{IEEEtrantools}
\usepackage{arydshln} 

\usepackage{tikz} 
\usetikzlibrary{positioning}
\usetikzlibrary{trees}
\usepackage{fancybox} 

\usepackage{multirow}
\usepackage{booktabs}

\usepackage{subfigure}

\usepackage{algorithm}

\begin{document}

\bibliographystyle{apalike}
\setcitestyle{round,colon,aysep={,},yysep={;}}

\title{Scenario-decomposition Solution Framework for \\ Nonseparable Stochastic Control Problems}
\author{
Xin Huang
\thanks{Department of Systems Engineering and Engineering Management, The Chinese University of Hong Kong, Shatin, Hong Kong (e-mails: huangxin@se.cuhk.edu.hk; zylong@cuhk.edu.hk).}
\and Duan Li
\thanks{Corresponding author. D. Li is with School of Data Science, City University of Hong Kong, Kowloon, Hong Kong (e-mail: dli226@cityu.edu.hk). The research of this author is supported by Hong Kong Research Grants Council under grant 11200219.}
\and Daniel Zhuoyu Long
\footnotemark[1]
}

\maketitle

\begin{abstract}
    When stochastic control problems do not possess separability and/or monotonicity, the dynamic programming pioneered by Bellman in 1950s fails to work as a time-decomposition solution method. Such cases have posted a great challenge to the control society in both theoretical foundation and solution methodologies for many years. With the help of the progressive hedging algorithm  proposed by Rockafellar and Wets in 1991, we develop a novel scenario-decomposition solution framework for stochastic control problems which could be nonseparable and/or non-monotonic, thus extending the reach of stochastic optimal control. We discuss then some of its promising applications, including online quadratic programming problems and dynamic portfolio selection problems with smoothing properties.
\end{abstract}
{\bf Keywords:} Nonseparable stochastic control, scenario decomposition, progressive hedging algorithm, online quadratic programming, dynamic portfolio selection.


\section{Introduction}\label{sec:introduction}
Stochastic control problems can be, in general, formulated as follows,
\begin{IEEEeqnarray*}{ccl}
~(\mathcal{P})~~ & \min\limits_{u_t, t=0, \ldots, T-1} & ~\mathbb{E}[J(x_0,u_0,x_1,u_1, \ldots, x_{T-1},u_{T-1},x_T)]\\
& {\rm s.t.} & ~x_{t+1} = f_t(x_t,u_t, \xi_t),\\
& & ~g_t(x_t,u_t) \leq 0,~g_T(x_T) \leq 0,\\
& & ~t = 0, 1, \ldots, T-1,
\end{IEEEeqnarray*}
where $x_t$ $\in$ $\mathbb{R}^m$ is the state with $x_0$ given, $u_t$ $\in$ $\mathbb{R}^n$ is the control, and $g_t(x_t,u_t)\leq 0$ and $g_T(x_T)\leq 0$ represent, respectively, the  running constraints on states and controls, and the constraint on the terminal state. Moreover, $\xi_t$ $\in$ $\mathbb{R}^p$ is a white noise vector, and $f_t:\mathbb{R}^m\times\mathbb{R}^n\times\mathbb{R}^p\rightarrow\mathbb{R}^m$ is the system dynamics. Thus, the system under consideration is of a Markovian property.
The performance measure $J$ is {\it backward separable} if there exist functions $\phi_{t}:$ $\mathbb{R}^m\times \mathbb{R}^n \times \mathbb{R}
\rightarrow \mathbb{R}$, $t=0, 1, \ldots, T-1$, and $\phi_{T}:$ $\mathbb{R}^m
\rightarrow \mathbb{R}$ such that
\begin{IEEEeqnarray*}{c}
J=\phi_{0}(x_0,u_0, \phi_{1}(x_1,u_1,
\phi_{2}(\ldots
\phi_{T-2}(x_{T-2},u_{T-2}, \phi_{T-1}(x_{T-1},u_{T-1},
\phi_T(x_T)))\ldots ))).
\end{IEEEeqnarray*}
The backward separable objective function $J$ is then said {\it backward monotonic} if for all $t$ = 0, 1, $\ldots$, $T-1$, the condition
\begin{align*}
\phi_{t}(\hat{x}_t,\hat{u}_t, \phi_{t+1}(\ldots
\phi_{T-1}(\hat{x}_{T-1},\hat{u}_{T-1}, \phi_T(\hat{x}_T))\ldots)) \leq
\phi_{t}(\tilde{x}_t,\tilde{u}_t, \phi_{t+1}(\ldots
\phi_{T-1}(\tilde{x}_{T-1},\tilde{u}_{T-1}, \phi_T(\tilde{x}_T))\ldots))
\end{align*}
implies the following: for
any triple $(x_{t-1},u_{t-1},\xi_{t-1})$ such that $\hat{x}_t = \tilde{x}_t = f_{t-1}(x_{t-1},u_{t-1}, \xi_{t-1})$, we have
\begin{align*}
&\phi_{t-1}(
x_{t-1},u_{t-1},\phi_{t}(\hat{x}_t,\hat{u}_t,\phi_{t+1}(\ldots
\phi_{T-1}(\hat{x}_{T-1},\hat{u}_{T-1}, \phi_T(\hat{x}_T))\ldots)))\leq \\
&\phi_{t-1}(
x_{t-1},u_{t-1},\phi_{t}(\tilde{x}_t,\tilde{u}_t,\phi_{t+1}(\ldots
\phi_{T-1}(\tilde{x}_{T-1},\tilde{u}_{T-1},\phi_T(\tilde{x}_T))\ldots))).
\end{align*}
When $(\mathcal{P})$ satisfies both the separability and the monotonicity as defined above, the celebrated dynamic programming (DP) \cite{Bellman1952} is a powerful time-decomposition solution approach, which is based on the principle of optimality.

There exist, however, a plenty of problems of interests that do not satisfy these fundamental requirements in DP. One notorious nonseparable case is the variance minimization problem (see \cite{white1974dynamic} and \cite{li2003variance}). The obstacle is mainly due to that the variance operation, unlike the expectation operator, does not satisfy the tower property along the time horizon. The variance minimization naturally emerges in the dynamic mean-variance (MV) portfolio selection problem. After many years of struggle, \cite{li2000optimal} finally solves it by embedding the original nonseparable problem into a family of separable auxiliary problems that are analytically solvable by DP. \cite{sniedovich1986c} and \cite{domingo1993experiments} in the early days consider nonseparable problems with the objective function of the form $h(u) = \psi(v(u), z(u))$, where both $v$ and $z$ are functions in additive forms w.r.t. stages. Under the assumption that $\psi$ is pseudo-concave w.r.t. its arguments, the authors of \cite{sniedovich1986c} and \cite{domingo1993experiments} develop the so-called C-programming to convert the primal problem into a separable version which could be handled by DP and report its applications in the variance minimization (see also \cite{sniedovich1987class}) and fractional programming (see also \cite{sniedovich1990solution}). \cite{carraway1990generalized} proposes a generalized DP for the multi-criteria optimization problem that violates the monotonicity. \cite{li1990new}, \cite{li1991extension}, and \cite{li1990multiple}
consider a class of nonseparable problems where the nonseparable objective function is a monotone function of several separable sub-objectives. Among these three papers, the first two deal with the deterministic cases, whereas the last one deals with the stochastic counterpart. They introduce the concept of $k$th-order separability and convert the primal nonseparable problem into a separable $k$-objective optimization problem which could be solved by the multi-objective DP \cite{Envelope}. They further develop conditions under which a specific Pareto solution is optimal to the original nonseparable problem. Moreover, \cite{li1997cost} investigates a nonseparable cost smoothing problem for the discrete-time deterministic linear-quadratic control.

Different from the above works, this paper aims to develop a novel solution framework through the \emph{scenario decomposition}, which is fundamentally distinct from the methods governed by \emph{time-decomposition}-based DP. Our solution framework relies on the progressive hedging algorithm (PHA) pioneered in \cite{rockafellar1991scenarios}. In contrast to DP, our PHA-oriented solution scheme can be applied to stochastic control problems which may not be separable and/or monotonic. We emphasize that PHA has not been fully recognized up to today for its powerful capability in dealing with the non-separability or non-monotonicity in stochastic control. We will further apply the newly-developed solution scheme to two nonseparable (thus non-tractable by DP) real-world applications: online quadratic programming (QP) and a novel variation of the dynamic portfolio selection problem with smoothing properties. Interestingly, the considered MV problem with smoothing feature could be embedded into a series of auxiliary problems that turn out to be a concrete type of our proposed online QP model.

The rest of the paper proceeds as follows. We build up in Section 2 the scenario-decomposition solution framework through adopting PHA on general stochastic control problems, where the information flow follows a tree structure. We then demonstrate its prominent application to the online QP problem in Section 3. On top of that, we also apply this solution methodology to  dynamic portfolio selection problems and their novel variations with smoothing features, and analyze experimental results in Section 4. Finally, we conclude the paper in Section 5.

\section{Solution Approach by Scenario Decomposition}\label{PHA_section}

We consider in this paper the problem $(\mathcal{P})$ with a Markovian system. As the most prominent feature of our new formulation, the objective function in general could be nonseparable and/or non-monotonic. On the other hand, we confine the structure of the information flow to a tree form, where the system randomness $\boldsymbol{\xi}=\{\xi_0,\xi_1,\ldots,\xi_{T-1}\}$ is realized stage by stage, and a series of realizations of $\xi_t$'s will form a \emph{scenario} of the tree, indexed by $i$. From the scenario analysis prospective, the dynamic stochastic control problem could be armed with a scenario tree in order to reflect its information flow for the underlying uncertainties.
Figure \ref{tree} exemplifies a specific three-stage tree structure, where $\boldsymbol{\xi}$ is realized successively from $\xi_0$ to $\xi_1$ and finally to $\xi_2$, thus leading to seven possible scenarios (paths of $\boldsymbol{\xi}$) in total. The number in each circle node represents a possible value of the disturbance realized right before that stage. Note that any parent node (starting from the square root node) could in general result in different numbers of children nodes. In contrast to DP whose effectiveness comes from the time decomposition, the solution power by PHA that we adopt in this paper roots in the scenario decomposition. Invented almost thirty years ago, PHA has been successfully applied to several application areas including power systems scheduling problems (see \cite{dos2009practical} among others) and water resource planning problems (see, e.g., \cite{carpentier2013long}). For more details on the general methodology of PHA, please refer to \cite{rockafellar1991scenarios}.

\begin{figure}[!t]
\begin{center}
\scalebox{1}{
\begin{tikzpicture}[
every node/.style={scale=0.8},
level 1/.style={sibling distance=4cm, level distance=1cm},
level 2/.style={sibling distance=2cm, level distance=1cm},
level 3/.style={sibling distance=2cm, level distance=1cm},
]
\node (Root) [rectangle,draw,minimum size=0.5cm] {} 
  child {node (xi0_1) [circle,draw,minimum size=0.8cm]  {$\xi_{0}$} 
    child {node (xi1_1) [circle,draw,minimum size=0.8cm]  {$\xi_{1}$} 
      child {node (xi2_1) [circle,draw,minimum size=0.8cm]  {$\xi_{2}$}} 
    }
    child {node (xi1_2) [circle,draw,minimum size=0.8cm]  {$\xi_{1}$} 
      child {node (xi2_2) [circle,draw,minimum size=0.8cm]  {$\xi_{2}$}} 
      child {node (xi2_3) [circle,draw,minimum size=0.8cm]  {$\xi_{2}$}}
    }
    child {node (xi1_3) [circle,draw,minimum size=0.8cm]  {$\xi_{1}$} 
      child {node (xi2_4) [circle,draw,minimum size=0.8cm]  {$\xi_{2}$}} 
    }
  }
  child {node (xi0_2) [circle,draw,minimum size=0.8cm]  {$\xi_{0}$} 
    child {node (xi1_4) [circle,draw,minimum size=0.8cm]  {$\xi_{1}$} 
      child {node (xi2_5) [circle,draw,minimum size=0.8cm]  {$\xi_{2}$}} 
    }
    child {node (xi1_5) [circle,draw,minimum size=0.8cm]  {$\xi_{1}$} 
      child {node (xi2_6) [circle,draw,minimum size=0.8cm]  {$\xi_{2}$}} 
      child {node (xi2_7) [circle,draw,minimum size=0.8cm]  {$\xi_{2}$}}
    }
  };
   \begin{scope}[every node/.style={right,scale=0.8}]
     \path (Root  -| xi2_7) ++(2mm,0) node (t0) {} ++(5mm,0) node {$t=0$};
     \path (xi0_1 -| xi2_7) ++(2mm,0) node (t1) {} ++(5mm,0) node {$t=1$};
     \path (xi1_1 -| xi2_7) ++(2mm,0) node (t2) {} ++(5mm,0) node {$t=2$};
     \path (xi2_1 -| xi2_7) ++(2mm,0) node (t3) {} ++(5mm,0) node {$T=3$};
   \end{scope}


\node [below = 0.1cm of xi2_1,align=center] {$i_1$};
\node [below = 0.1cm of xi2_2,align=center] {$i_2$};
\node [below = 0.1cm of xi2_3,align=center] {$i_3$};
\node [below = 0.1cm of xi2_4,align=center] {$i_4$};
\node [below = 0.1cm of xi2_5,align=center] {$i_5$};
\node [below = 0.1cm of xi2_6,align=center] {$i_6$};
\node [below = 0.1cm of xi2_7,align=center] {$i_7$};

\end{tikzpicture}%
}
\end{center}
\caption{A scenario tree with three stages and seven scenarios.}
\label{tree}
\end{figure}
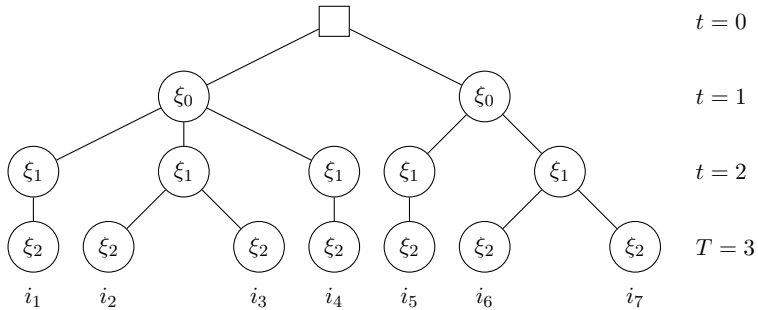

Let us denote by $\mathcal{I}$ the scenario set which consists of all possible scenarios, and denote by $\boldsymbol{\xi}^i=\{\xi_0^i,\xi_1^i,\ldots,\xi_{T-1}^i\}$ the realizations of disturbance under the scenario $i\in\mathcal{I}$. Assuming the occurring probability of scenario $i$ to be $\rho_i$ that is fixed at time $0$, we can rewrite the objective of $(\mathcal{P})$ as $\min\sum_{i\in\mathcal{I}}\rho_i J_i$, where $J_i$ denotes the sub-objective under $\boldsymbol{\xi}^i$. Then it is natural to decompose problem $(\mathcal{P})$ into a family of \emph{scenario subproblems} and  consider the following individual scenario subproblem for each $i\in\mathcal{I}$,
\begin{IEEEeqnarray*}{ccl}
  ~(\mathcal{P}^i)~~ & \min_{u_t,\forall t} & ~J_i = J(x_0^i,u_0,x_1^i,u_1,\ldots,x_{T-1}^i,u_{T-1},x_T^i) \\
  & \rm{s.t.} & ~x_{t+1}^i = f_t(x_t^i,u_t, \xi_t^i),~x_0^i=x_0,\\
  && ~g_t(x_t^i,u_t) \leq 0,~g_T(x_T^i) \leq 0,\\
  && ~t = 0, 1, \ldots, T-1,
\end{IEEEeqnarray*}
which is a {\it deterministic} optimal control problem, and should be much easier to solve than the original stochastic one. In this paper, we further assume that each $(\mathcal{P}^i)$ is convex w.r.t. the control variable $\mathbf{u}=(u_0',u_1',\ldots,u_{T-1}')'\in\mathbb{R}^{nT}$. Although the optimal solution of $(\mathcal{P}^i)$ satisfies all the \emph{admissible} constraints of the primal problem $(\mathcal{P})$, it is \emph{not implementable} in reality, since we have ``stolen'' the \emph{future} information (i.e., the future realization of $\boldsymbol{\xi}$) when solving each scenario subproblem at time $0$. In other words, the scenario-specific solutions violate the so-called \emph{nonanticipative} constraint which is either explicitly or implicitly implied in any stochastic control problem. To force any admissible solution to meet nonanticipativity, the \emph{scenario bundles}, as a partition of $\mathcal{I}$, are formed at each time according to the scenario tree of the underlying problem. Graphically speaking, scenarios passing through each node at a certain time stage are grouped together to form a bundle. In Figure \ref{tree}, for instance, at time 0 all the scenarios form a single bundle that is the scenario set itself and we denote this partition by $\mathcal{I}_0=\{\mathcal{I}_{0,1}\}=\{\{i_1,\ldots,i_7\}\}$; and when $t=1$ we have two bundles together to form $\mathcal{I}_1=\{\mathcal{I}_{1,1},\mathcal{I}_{1,2}\}=\{\{i_1,\ldots,i_4\},\{i_5,\ldots,i_7\}\}$; and finally for $t=2$ we have five bundles to form the partition of $\mathcal{I}$ at that time, i.e., $\mathcal{I}_2=\{\mathcal{I}_{2,1},\mathcal{I}_{2,2},\mathcal{I}_{2,3},\mathcal{I}_{2,4},\mathcal{I}_{2,5}\}
=\{\{i_1\},\{i_2,i_3\},\{i_4\},\{i_5\},\{i_6,i_7\}\}$. The nonanticipativity naturally requires any \emph{implementable} policy to react the same to all indifferent scenarios (the scenarios from the same bundle), and this is achieved by taking conditional expectations on the scenario-specific solutions from the related bundle. More specifically, the implementable control at time $t$, if the scenario $i$ occurs, is computed through
\begin{align}
  \hat{u}^{i}_t= \sum_{j\in \mathcal{I}_{t,l}}\frac{\rho_j}{\sum_{j'\in \mathcal{I}_{t,l}}\rho_{j'}}u_t^{j},~i\in \mathcal{I}_{t,l},~l=1,\ldots,|\mathcal{I}_t|,
  \label{conditional_expectation}
\end{align}
where $u_t^{j}$ is the scenario-$j$-based admissible control at time $t$, and $|\mathcal{I}_t|$ is the number of scenario bundles in the partition $\mathcal{I}_t$. Note that $|\mathcal{I}_t|$ determines the number implementable controls corresponding to different realizations at that time. In fact, the above procedure in \eqref{conditional_expectation}
can be characterized in a linear transformation $\hat{\mathbf{u}}_t=\mathbf{T}_t\mathbf{u}_t$, where $\hat{\mathbf{u}}_t=((\hat{{u}}_t^1)',(\hat{{u}}_t^2)',\ldots,(\hat{{u}}_t^{|\mathcal{I}|})')'
\in\mathbb{R}^{n|\mathcal{I}|}$, ${\mathbf{u}}_t=(({{u}}_t^1)',({{u}}_t^2)',\ldots,({{u}}_t^{|\mathcal{I}|})')'
\in\mathbb{R}^{n|\mathcal{I}|}$, and the projection matrix $\mathbf{T}_t$ can be easily build up by scenario probabilities based on the structure of $\mathcal{I}_t$.
Then the overall linear mapping is
\begin{align*}
  \left(
  \begin{array}{c}
    \hat{\mathbf{u}}_0\\
    \hat{\mathbf{u}}_1\\
    \vdots\\
    \hat{\mathbf{u}}_{T-1}
  \end{array}
  \right)=
  \left(
  \begin{array}{cccc}
    \mathbf{T}_0 & 0 & \cdots & 0 \\
    0 & \mathbf{T}_1 & \cdots & 0 \\
    \vdots & \vdots & \ddots & \vdots \\
    0 & 0 & \cdots & \mathbf{T}_{T-1}
  \end{array}
  \right)
  \left(
  \begin{array}{c}
    {\mathbf{u}}_0\\
    {\mathbf{u}}_1\\
    \vdots\\
    {\mathbf{u}}_{T-1}
  \end{array}
  \right).
\end{align*}

The beauty of PHA lies in its augmented Lagrangian formulation that progressively aggregates the scenario-specific solutions into an implementable one and forces them to converge to the optimal solution of the primal problem $(\mathcal{P})$, which are both admissible and implementable.
More precisely, it deals with an augmented Lagrangian problem at each iteration $\nu=0,1,\ldots$, which is constructed by adding a linear Lagrangian term and a quadratic penalty term to the scenario-specific objective function in order to penalize any utilization of the anticipative information from the future. More precisely, we solve the following augmented Lagrangian problem in the $\nu$th iteration for each $i\in\mathcal{I}$,
\begin{IEEEeqnarray*}{ccl}
  ~(\mathcal{P}^{i,\nu})~~ & \min_{\mathbf{u}} & ~J(x_0^i,u_0,x_1^i,u_1,\ldots,x_{T-1}^i,u_{T-1},x_T^i)
  +\mathbf{u}'\mathbf{w}^{i,\nu}+\frac{1}{2}\alpha|\mathbf{u}-\hat{\mathbf{u}}^{i,\nu}|_2^2, \\
  & \rm{s.t.} & ~x_{t+1}^i = f_t(x_t^i,u_t, \xi^i_t),~x_0^i=x_0,\\
  && ~ g_t(x_t^i,u_t) \leq 0,~g_T(x_T^i) \leq 0,\\
  && ~ t = 0, 1, \ldots, T-1,
\end{IEEEeqnarray*}
where we define, for compactness, $\mathbf{u}=(u_0',u_1',\ldots,u_{T-1}')'\in\mathbb{R}^{nT}$ as the overall control vector, and
\begin{align}
  \hat{\mathbf{u}}^{i,\nu}=((\hat{u}^{i,\nu}_0)',(\hat{u}^{i,\nu}_1)',\ldots,(\hat{u}^{i,\nu}_{T-1})')'
  \in\mathbb{R}^{nT}
\end{align}
is a given implementable control for $(\mathcal{P}^{i,\nu})$. Let us denote the optimal solution of $(\mathcal{P}^{i,\nu})$ by
\begin{align}
  \mathbf{u}^{i,\nu+1}=((u^{i,\nu+1}_0)',(u^{i,\nu+1}_1)',\ldots,(u^{i,\nu+1}_{T-1})')',
  \in\mathbb{R}^{nT}
  \label{nu+1_sol}
\end{align}
which is a new scenario-based solution. We then aggregate all $\mathbf{u}^{i,\nu+1}$, $i\in\mathcal{I}$, into a new implementable control, denoted by
\begin{align}
  \hat{\mathbf{u}}^{i,\nu+1}=((\hat{u}^{i,\nu+1}_0)',(\hat{u}^{i,\nu+1}_1)',\ldots,
  (\hat{u}^{i,\nu+1}_{T-1})')'\in\mathbb{R}^{nT},
  \label{nu+1_hat}
\end{align}
through the componentwise calculations of (\ref{conditional_expectation}), or in the following compact way: we first gather $u_t^{i,\nu+1}$ of all $i$ to form
\begin{align}
  \mathbf{u}_t^{\nu+1}=((u_t^{1,\nu+1})',(u_t^{2,\nu+1})',\ldots,(u_t^{|\mathcal{I}|,\nu+1})')'
  \in\mathbb{R}^{n|\mathcal{I}|},
  \label{nu+1_project1}
\end{align}
and conduct the transformation $\hat{\mathbf{u}}_t^{\nu+1}=\mathbf{T}_t\mathbf{u}_t^{\nu+1}$, where
\begin{align}
  \hat{\mathbf{u}}_t^{\nu+1}=((\hat{u}_t^{1,\nu+1})',(\hat{u}_t^{2,\nu+1})',\ldots,
  (\hat{u}_t^{|\mathcal{I}|,\nu+1})')'\in\mathbb{R}^{n|\mathcal{I}|};
  \label{nu+1_project2}
\end{align}
and this is done for every $t=0,1,\ldots,T-1$. We then pick up the $i$th component of $\hat{\mathbf{u}}_t^{\nu+1}$, $\hat{u}_t^{i,\nu+1}$, for all $t$, to serve as $\hat{\mathbf{u}}^{i,\nu+1}$ in $(\mathcal{P}^{i,\nu+1})$. When $\nu=0$, all the initial $\hat{\mathbf{u}}^{i,0}$, $i\in\mathcal{I}$, are attained from $\mathbf{u}^{i,0}$, $i\in\mathcal{I}$, following the above procedure, where $\mathbf{u}^{i,0}$ could be selected as the optimal solution of $(\mathcal{P}^i)$. In $(\mathcal{P}^{i,\nu})$,  the  penalty parameter $\alpha>0$ is predetermined, and the Lagrangian multiplier $\mathbf{w}^{i,\nu}=((w^{i,\nu}_0)',(w^{i,\nu}_1)',\ldots,(w^{i,\nu}_{T-1})')'\in\mathbb{R}^{nT}$,
for every $i$, satisfies the recursion below,
\begin{align}
  \mathbf{w}^{i,\nu+1}=\mathbf{w}^{i,\nu}+\alpha (\mathbf{u}^{i,\nu+1}-\hat{\mathbf{u}}_t^{i,\nu+1}),
  \label{w_update}
\end{align}
where $\mathbf{w}^{i,0}$ is set at zero. The solution process repeats until a stopping criterion is satisfied. We now provide the convergence result as follows.

\begin{thm}[Convergence of PHA, \cite{rockafellar1991scenarios}]
If all the scenario subproblems $(\mathcal{P}^{i})$ are convex w.r.t. $\mathbf{u}$ and have been solved exactly, and $\{\mathbf{u}: g_t(x_t,u_t) \leq 0,\forall t\}$ is a convex set under any $x_t$, then the sequence $\{\hat{\mathbf{u}}^{i,\nu+1}\}_\nu$, generated by $(\mathcal{P}^{i,\nu})$, $\nu=0,1,\ldots$, converges to the real optimal $\mathbf{u}^{i,*}$, $i\in\mathcal{I}$, of the primal problem $(\mathcal{P})$. And on the other hand, the sequence $\{\mathbf{w}^{i,\nu+1}\}_\nu$ converges to $\mathbf{w}^{i,*}$, which is also known as the shadow price for each scenario $i$ of the problem. Moreover, the solution quality is guaranteed continuously improved,
in the sense that
\begin{align}
\sum_{i\in\mathcal{I}}\rho_i\left(|\hat{\mathbf{u}}^{i,\nu+1}-\mathbf{u}^{i,*}|_2^2
+\frac{1}{\alpha^2}|\mathbf{w}^{i,\nu+1}-\mathbf{w}^{i,*}|_2^2\right) \leq
\sum_{i\in\mathcal{I}}\rho_i\left(|\hat{\mathbf{u}}^{i,\nu}-\mathbf{u}^{i,*}|_2^2
+\frac{1}{\alpha^2}|\mathbf{w}^{i,\nu}-\mathbf{w}^{i,*}|_2^2\right),
\end{align}
and the equality is finally achieved when $(\hat{\mathbf{u}}^{i,\nu+1},\mathbf{w}^{i,\nu+1})=(\mathbf{u}^{i,*},\mathbf{w}^{i,*})$ for some $\nu$.
\label{PHA_thm}
\end{thm}

\section{Online Quadratic Programming}\label{QP_section}

Quadratic programming (QP) is a fundamental subject in mathematical programming with wide spectra of applications in various fields, including business and finance (see \cite{QP_Survey} for a survey). Although QP has been investigated broadly and deeply, almost all of the studies up to today have been confined in a deterministic framework. Recently, \cite{Online_LP} studies the online linear programming (LP), where the constraint matrix is revealed column by column along with the corresponding coefficients in the objective function. In this section, we will extend the online programming from online LP to online QP and solve it by our newly proposed solution scheme introduced in Section \ref{PHA_section}. More precisely, we consider an online version of a general QP,
\begin{IEEEeqnarray*}{ccl}
~(\mathcal{Q})~~ & \min\limits_{u_t,\forall t} & ~\mathbb{E}\left[\sum\nolimits_{i,j=0}^{T} \frac{1}{2}x_i' Q_{ij} x_j + \sum\nolimits_{t=0}^{T}x_t'c_{t}
+ \sum\nolimits_{i,j=0}^{T-1} \frac{1}{2}u_i' R_{ij} u_j + \sum\nolimits_{t=0}^{T-1}u_t'd_{t}\right] \\
& \rm{s.t.} & ~x_{t+1}=A_tx_t+B_t u_t+\xi_t,~t=0,1,\ldots,T-1,
\end{IEEEeqnarray*}
where $x_t\in\mathbb{R}^m$ is the state with $x_0$ given, $u_t\in\mathbb{R}^n$ is the control, and $\xi_t \in\mathbb{R}^m$ is the system randomness at time $t$ following some discrete distribution $D_{\xi_t}$ with $|D_{\xi_t}|$ possible outcomes and the probability $\pi_t^k$ for each $k=1,2,\ldots,|D_{\xi_t}|$. We further assume that $\xi_t$'s are independent across time stages. Therefore, there are in total $\prod_{t=0}^{T-1}|D_{\xi_t}|$ scenarios for this $T$-period problem, and each scenario reflects a path of $\xi_t$'s along the time horizon, and the scenario probability $\rho_i$ is calculated by the product of the involved $\pi_t^k$'s. The assumptions on the coefficients will be stated later. Note that the system disturbance $\xi_t$ is realized after the decision is made at time $t$. To see the \emph{online} nature of $(\mathcal{Q})$, we can aggregate all the constraints into the following compact form,
\begin{align}
  &\left(
  \begin{array}{cccccc}
  I_m & 0  & 0  & \cdots & 0 & 0\\
  -A_1 & I_m & 0 & \cdots & 0 & 0 \\
  0 & -A_2 & I_m & \cdots & 0 & 0 \\
  \vdots & \vdots & \ddots & \ddots & \vdots & \vdots\\
  0 & 0 & 0 & \cdots & -A_{T-1} & I_m
  \end{array}
  \right)\left(
  \begin{array}{c}
    x_1\\
    x_2\\
    x_3\\
    \vdots\\
    x_T
  \end{array}
  \right)
  -\left(
  \begin{array}{cccc}
  B_{0} & 0 & \cdots & 0\\
  0 & B_{1} & \cdots & 0 \\
  \vdots & \vdots & \ddots & \vdots\\
  0 & 0 & \ldots & B_{T-1}
  \end{array}
  \right)\left(
  \begin{array}{c}
    u_0\\
    u_1\\
    \vdots\\
    u_{T-1}
  \end{array}
  \right)\nonumber\\
  &=\left(
  \begin{array}{c}
    A_0x_0\\
    0\\
    \vdots\\
    0
  \end{array}
  \right)+\left(
  \begin{array}{c}
    \xi_0\\
    \xi_1\\
    \vdots\\
    \xi_{T-1}
  \end{array}
  \right),
  \label{QP_constraint}
\end{align}
where $I_m$ is an identity matrix of size $m$. It becomes clear now that the right hand side of the constraints (\ref{QP_constraint}) is fully uncertain at time $0$, and it becomes partially deterministic when time evolves. For instance, at time 1 (before $u_1$ is made), only the first constraint becomes deterministic. In general, at time $t$ (before $u_t$ is determined), the first $t$ constraints are realized. Although we can observe the states when the system randomness is gradually achieved, it is often the case that we need to make optimal decisions before that happens. On the other hand, the objective function of $(\mathcal{Q})$, in general, includes cross terms on $x_t$'s and $u_t$'s in terms of $t$, respectively. These interactions among time stages make $(\mathcal{Q})$ a concrete nonseparable instance of $(\mathcal{P})$. Let us take a deeper look on its compact form and make some assumptions on its coefficients,
\begin{IEEEeqnarray*}{ccl}
  ~(\mathcal{Q}_c)~~&\min_{\mathbf{u}} &~\mathbb{E}\left[\frac{1}{2}\mathbf{x}'\mathbf{Q}\mathbf{x}
  +\mathbf{x}'\mathbf{c}+\frac{1}{2}\mathbf{u}'\mathbf{R}\mathbf{u}+\mathbf{u}'\mathbf{d}\right]\\
  & {\rm s.t.} & ~\mathbf{x}=\mathbf{A}+\mathbf{B}\mathbf{u}+\mathbf{C}\boldsymbol{\xi},
\end{IEEEeqnarray*}
where $\mathbf{x}=(x_0',x_1',\ldots,x_{T-1}',x_T')'\in\mathbb{R}^{m(T+1)}$, $\mathbf{u}=(u_0',u_1',\ldots,u_{T-1}')'\in\mathbb{R}^{nT}$, and $\boldsymbol{\xi}=(\xi_0',\xi_1',\ldots,\xi_{T-1}')'\in\mathbb{R}^{mT}$, and the coefficient matrices are given by
\begin{align*}
  &\mathbf{Q}=\left(
  \begin{array}{cccc}
    Q_{00} & Q_{01} & \cdots & Q_{0T} \\
    Q_{10} & Q_{11} & \cdots & Q_{1T} \\
    \vdots & \vdots & \ddots & \vdots \\
    Q_{T0} & Q_{T1} & \cdots & Q_{TT}
  \end{array}
  \right)\in\mathbb{R}^{m(T+1)\times m(T+1)},\\
  &\mathbf{R}=\left(
  \begin{array}{cccc}
    R_{00} & R_{01} & \cdots & R_{0(T-1)} \\
    R_{10} & R_{11} & \cdots & R_{1(T-1)} \\
    \vdots & \vdots & \ddots & \vdots \\
    R_{(T-1)0} & R_{(T-1)1} & \cdots & R_{(T-1)(T-1)}
  \end{array}
  \right)\in\mathbb{R}^{nT\times nT},
\end{align*}
$\mathbf{c}=(c_0',c_1',\ldots,c_T')'\in\mathbb{R}^{m(T+1)}$,  $\mathbf{d}=(d_0',d_1',\ldots,d_{T-1}')'\in\mathbb{R}^{nT}$,
\begin{align*}
  &\mathbf{A}=\left(x_0',(A_0x_0)',(A_1A_0x_0)',\ldots,(\prod_{t=0}^{T-1}A_tx_0)'\right)'
  \in\mathbb{R}^{m(T+1)},
\end{align*}
\begin{align*}
  &\mathbf{B}=\left(
  \begin{array}{cccc}
    {0} & {0} & \ldots & {0}\\
    B_0 & {0} & \ldots & {0}\\
    A_1B_0 & B_1 & \ldots & {0}\\
    \vdots & \vdots & \ddots & \vdots \\
    (\prod\limits_{t=1}^{T-1}A_t)B_0 & (\prod\limits_{t=2}^{T-1}A_t)B_1 &
    \ldots & B_{T-1}
  \end{array}
  \right)\in\mathbb{R}^{m(T+1)\times nT},
\end{align*}
and
\begin{align*}
  &\mathbf{C}=\left(
  \begin{array}{cccc}
    {0} & {0} & \ldots & {0}\\
    I_m & {0} & \ldots & {0}\\
    A_1 & I_m & \ldots & {0}\\
    \vdots & \vdots & \ddots & \vdots \\
    \prod_{t=1}^{T-1}A_t & \prod_{t=2}^{T-1}A_t &
    \ldots & I_m
  \end{array}
  \right)\in\mathbb{R}^{m(T+1)\times mT}.
\end{align*}
\begin{asmp}
  The matrices $\mathbf{Q}$ and $\mathbf{R}$ are positive semidefinite.
  \label{QR_assumption}
\end{asmp}
The conventional stochastic linear-quadratic (LQ) problem turns out to be a special case of $(\mathcal{Q}_c)$ in which both $\mathbf{Q}$ and $\mathbf{R}$ are diagonal block matrices,  and are positive semidefinite and positive definite, respectively.
Under Assumption \ref{QR_assumption}, $(\mathcal{Q}_c)$ is solvable by our proposed scenario-decomposition scheme, as each scenario subproblem
\begin{IEEEeqnarray*}{ccl}
  ~(\mathcal{Q}_c^i)~~&\min_{\mathbf{u}} &~\frac{1}{2}\mathbf{x}'\mathbf{Q}\mathbf{x}
  +\mathbf{x}'\mathbf{c}+\frac{1}{2}\mathbf{u}'\mathbf{R}\mathbf{u}+\mathbf{u}'\mathbf{d}\\
  & {\rm s.t.} & ~\mathbf{x}=\mathbf{A}+\mathbf{B}\mathbf{u}+\mathbf{C}\boldsymbol{\xi}^i,
\end{IEEEeqnarray*}
is convex w.r.t. the decision variable $\mathbf{u}$. If $\mathbf{R}$ is further positive definite, we have the optimal solution to $(\mathcal{Q}_c^i)$, denoted by $\mathbf{u}^{i,0}$, in an analytical form given by
\begin{align}
  \mathbf{u}^{i,0}=-(\mathbf{B}'\mathbf{Q}\mathbf{B}+\mathbf{R})^{-1}
  [\mathbf{B}'\mathbf{Q}(\mathbf{A}+\mathbf{C}\boldsymbol{\xi}^i)+\mathbf{B}'\mathbf{c}+\mathbf{d}].
  \label{Q_scenario_sol}
\end{align}
Note again that the optimal solution to the $i$th scenario problem, $\mathbf{u}^{i,0}$, is not the optimal result to the primal problem $(\mathcal{Q}_c)$, even not a feasible one since it violates the nonanticipative constraint. We now apply the scenario-decomposition solution approach to $(\mathcal{Q}_c)$. More precisely, let us consider at iteration $\nu=0,1,\ldots$, the following augmented Lagrangian problem for each scenario $i$,
\begin{IEEEeqnarray*}{ccl}
  ~(\mathcal{Q}_c^{i,\nu})~~&\min_{\mathbf{u}}&~\frac{1}{2}\mathbf{x}'\mathbf{Q}\mathbf{x}
  +\mathbf{x}'\mathbf{c}+\frac{1}{2}\mathbf{u}'\mathbf{R}\mathbf{u}+\mathbf{u}'\mathbf{d}
  +\mathbf{u}'\mathbf{w}^{i,\nu}+\frac{1}{2}\alpha|\mathbf{u}-\hat{\mathbf{u}}^{i,\nu}|_2^2\\
  &{\rm s.t.}&~\mathbf{x}=\mathbf{A}+\mathbf{B}\mathbf{u}+\mathbf{C}\boldsymbol{\xi}^i,
\end{IEEEeqnarray*}
for a given implementable policy $\hat{\mathbf{u}}^{i,\nu}$ and a Lagrangian multiplier $\mathbf{w}^{i,\nu}$ (note that when $\nu=0$, $\hat{\mathbf{u}}^{i,0}$ is set at the implementable solution attained from $\mathbf{u}^{i,0}$, the optimal solution of $(\mathcal{Q}^i_c)$, and $\mathbf{w}^{i,0}$ is set as a zero vector). This time, due to the newly-added quadratic term on $\mathbf{u}$ in the objective, the optimal solution of $(\mathcal{Q}_c^{i,\nu})$, denoted by $\mathbf{u}^{i,\nu+1}$, is always given analytically by
\begin{align}
  \mathbf{u}^{i,\nu+1}={}&-(\mathbf{B}'\mathbf{Q}\mathbf{B}+\mathbf{R}+\alpha I_{nT})^{-1}
  \left[\mathbf{B}'\mathbf{Q}(\mathbf{A}+\mathbf{C}\boldsymbol{\xi}^i)+\mathbf{B}'\mathbf{c}+\mathbf{d}
  +\mathbf{w}^{i,\nu}-\alpha\hat{\mathbf{u}}^{i,\nu}\right],
  \label{Q_aug_sol}
\end{align}
where $I_{nT}$ is an $nT$-by-$nT$ identity matrix. Note that the explicit recursions in \eqref{Q_aug_sol} help us saving efforts when we deal with the iterative augmented Lagrangian problems. Therefore, the algorithm for this type of application is quite efficient. We then calculate $\hat{\mathbf{u}}^{i,\nu+1}$, the implementable solution for the next iteration, based on \eqref{conditional_expectation} or following the same procedure shown from \eqref{nu+1_sol} to \eqref{nu+1_project2}, and update $\mathbf{w}^{i,\nu+1}$ according to \eqref{w_update}. In practice, we could select the following condition as our stopping criterion,
\begin{align}
  \sum_{i\in\mathcal{I}}\rho_i\left(|\hat{\mathbf{u}}^{i,\nu+1}-\hat{\mathbf{u}}^{i,\nu}|_2^2 + \frac{1}{\alpha^2}|\mathbf{w}^{i,\nu+1}-\mathbf{w}^{i,\nu}|_2^2\right) \leq \epsilon,
  \label{Q_stopping criterion}
\end{align}
for a sufficiently small tolerance $\epsilon>0$. The set of implementable controls $\{\hat{\mathbf{u}}^{i,\nu+1}:i\in\mathcal{I}\}$ that satisfies this stopping rule is chosen as the optimal solution to $(\mathcal{Q})$ or $(\mathcal{Q}_c)$, which is denoted by $\{\hat{\mathbf{u}}^{i,\infty}:i\in\mathcal{I}\}$.

\begin{exam}\label{onlineQ_eg}
Let us consider an illustrative problem with a scalar state ($m=1$), a two-dimensional control ($n=2$), and a planning horizon of $T=3$. The system parameters are simply given by $A_t=1$ and $B_t=(1,1)$ for all $t$, whereas $\mathbf{Q}=(Q_{ij})_{i,j=0}^T$ and $\mathbf{R}=(R_{ij})_{i,j=0}^{T-1}$ in the performance measure are randomly generated as follows,
\begin{align*}
&\mathbf{Q}=\left(
\begin{array}{c:c:c:c}
       2.4512 & 1.0930 & 1.0243 & 1.8873 \\
       \hdashline
       1.0930 & 0.7852 & 0.2319 & 1.0027 \\
       \hdashline
       1.0243 & 0.2319 & 0.7276 & 0.5147 \\
       \hdashline
       1.8873 & 1.0027 & 0.5147 & 1.7188
\end{array}
\right),
\\
&\mathbf{R}=\left(
\begin{array}{cc:cc:cc}
       1.3281 & 1.4932 & 1.2903 & 0.7788 & 1.0149 & 1.0774 \\
       1.4932 & 2.6110 & 2.2984 & 1.3315 & 1.3902 & 2.3629 \\
       \hdashline
       1.2903 & 2.2984 & 2.7214 & 1.7258 & 1.7339 & 2.6799 \\
       0.7788 & 1.3315 & 1.7258 & 1.3102 & 1.0305 & 1.6583 \\
       \hdashline
       1.0149 & 1.3902 & 1.7339 & 1.0305 & 1.3073 & 1.6208 \\
       1.0774 & 2.3629 & 2.6799 & 1.6583 & 1.6208 & 2.9734
\end{array}
\right).
\end{align*}
The above two matrices are positive semidefinite and positive definite, respectively. To have a positive definite $\mathbf{R}$ in this example is for the purpose of comparison with the classical stochastic LQ control. Furthermore, $\mathbf{c}$ and $\mathbf{d}$  are set to be zero vectors for simplicity. The white system disturbance $\xi_t$ is modeled by a two-point distribution at each time $t$ with $D_{\xi_t}=\{1,-1\}$ and equal probability. Hence this is simply a binomial scenario tree as shown in Figure \ref{Q_eg_tree}, where the possible realizations of $\xi_t$'s at different time stages and under different scenarios are listed next to the related circle nodes. The total number of scenarios is $|\mathcal{I}|=8$ with the scenario probability $\rho_i=1/8$ for every $i\in\mathcal{I}$. The partitions of the scenario set, $\mathcal{I}_t$'s, together with scenario bundles at each time, $\mathcal{I}_{t,l}$'s, are easily recognized: $\mathcal{I}_0=\{\mathcal{I}_{0,1}\}=\{\{i_1,\ldots,i_{8}\}\}$; $\mathcal{I}_1=\{\mathcal{I}_{1,1},\mathcal{I}_{1,2}\}=\{\{i_1,\ldots,i_{4}\},\{i_{5},\ldots,i_{8}\}\}$; and finally $\mathcal{I}_2=\{\mathcal{I}_{2,1},\ldots,\mathcal{I}_{2,4}\}
=\{\{i_1,i_2\},\ldots,\{i_{7},i_{8}\}\}$.  Suppose the system starts from $x_0=1$. The optimal controls $\hat{\mathbf{u}}^{i,\infty}=((\hat{u}_0^{i,\infty})',(\hat{u}_1^{i,\infty})',(\hat{u}_2^{i,\infty})')'$, $i\in\mathcal{I}$, solved by the scenario-decomposition scheme in MATLAB for the above online QP problem,  are displayed (rounding in two decimals) beneath the corresponding nodes in Figure \ref{Q_eg_tree}. We next keep only diagonal blocks and set others to be zeros in the above $\mathbf{Q}$ and $\mathbf{R}$ and investigate the resulted standard stochastic LQ problem using both PHA and DP. We find that the optimal controls obtained from both methods coincide with each other. This exercise numerically demonstrates equivalent solution powers to certain degrees from both time decomposition and scenario decomposition approaches when both are applied to the separable and monotone stochastic control problems with convex scenario subproblems.

\begin{figure}[!t]
\begin{center}
\scalebox{0.9}{
\begin{tikzpicture}[
every node/.style={scale=0.8},
level 1/.style={sibling distance=8cm, level distance=1cm},
level 2/.style={sibling distance=4cm, level distance=1cm},
level 3/.style={sibling distance=3cm, level distance=2cm},
]
\node (Root) [rectangle,draw,minimum size=0.5cm] {} 
  child {node (xi0_1) [circle,draw,minimum size=0.8cm]  {1} 
    child {node (xi1_1) [circle,draw,minimum size=0.8cm]  {1} 
      child {node (xi2_1) [circle,draw,minimum size=0.8cm]  {1}} 
      child {node (xi2_2) [circle,draw,minimum size=0.8cm]  {-1}}
    }
    child {node (xi1_2) [circle,draw,minimum size=0.8cm]  {-1} 
      child {node (xi2_3) [circle,draw,minimum size=0.8cm]  {1}} 
      child {node (xi2_4) [circle,draw,minimum size=0.8cm]  {-1}}
    }
  }
  child {node (xi0_2) [circle,draw,minimum size=0.8cm]  {-1} 
    child {node (xi1_3) [circle,draw,minimum size=0.8cm]  {1} 
      child {node (xi2_5) [circle,draw,minimum size=0.8cm]  {1}} 
      child {node (xi2_6) [circle,draw,minimum size=0.8cm]  {-1}}
    }
    child {node (xi1_4) [circle,draw,minimum size=0.8cm]  {-1} 
      child {node (xi2_7) [circle,draw,minimum size=0.8cm]  {1}} 
      child {node (xi2_8) [circle,draw,minimum size=0.8cm]  {-1}}
    }
  };
   \begin{scope}[every node/.style={right,scale=0.8}]
     \path (Root  -| xi2_8) ++(2mm,0) node {} ++(5mm,0) node {$t=0$};
     \path (xi0_1 -| xi2_8) ++(2mm,0) node {} ++(5mm,0) node {$t=1$};
     \path (xi1_1 -| xi2_8) ++(2mm,0) node {} ++(5mm,0) node {$t=2$};
     \path (xi2_1 -| xi2_8) ++(2mm,0) node {} ++(5mm,0) node {$T=3$};
   \end{scope}

\path (xi0_1) -- (xi0_2) node [midway,align=center] {
$\footnotesize\hat u_0^{i,\infty}=\left(
\begin{aligned}
  0.37 \\ 
  -1.83 
\end{aligned}
\right)$,\\$\footnotesize\forall i\in\mathcal{I}_{0,1}$};

\path (xi1_1) -- (xi1_2) node [below = 0.02cm of xi0_1,align=center] {
$\footnotesize\hat u_1^{i,\infty}=\left(
\begin{aligned}
  3.58 \\ 
  -5.01 
\end{aligned}
\right)$,\\$\footnotesize\forall i\in\mathcal{I}_{1,1}$};

\path (xi1_3) -- (xi1_4) node [below = 0.02cm of xi0_2,align=center] {
$\footnotesize\hat u_1^{i,\infty}=\left(
\begin{aligned}
  1.60 \\ 
  -1.50 
\end{aligned}
\right)$,\\$\footnotesize\forall i\in\mathcal{I}_{1,2}$};

\node (u_1_1) [below = 0.8cm of xi1_1,align=center] {
$\footnotesize\hat u_2^{i,\infty}=\left(
\begin{aligned}
  -3.23 \\ 
  1.87 
\end{aligned}
\right)$,\\$\footnotesize\forall i\in\mathcal{I}_{2,1}$};

\node (u_1_2) [below = 0.8cm of xi1_2,align=center] {
$\footnotesize\hat u_2^{i,\infty}=\left(
\begin{aligned}
  -1.25 \\ 
  1.41 
\end{aligned}
\right)$,\\$\footnotesize\forall i\in\mathcal{I}_{2,2}$};

\node (u_1_3) [below = 0.8cm of xi1_3,align=center] {
$\footnotesize\hat u_2^{i,\infty}=\left(
\begin{aligned}
  -1.60 \\ 
  1.25 
\end{aligned}
\right)$,\\$\footnotesize\forall i\in\mathcal{I}_{2,3}$};

\node (u_1_4) [below = 0.8cm of xi1_4,align=center] {
$\footnotesize\hat u_2^{i,\infty}=\left(
\begin{aligned}
  0.38 \\ 
  0.79 
\end{aligned}
\right)$,\\$\footnotesize\forall i\in\mathcal{I}_{2,4}$};

\node [below = 0.1cm of xi2_1,align=center] {$i_1$};
\node [below = 0.1cm of xi2_2,align=center] {$i_2$};
\node [below = 0.1cm of xi2_3,align=center] {$i_3$};
\node [below = 0.1cm of xi2_4,align=center] {$i_4$};
\node [below = 0.1cm of xi2_5,align=center] {$i_5$};
\node [below = 0.1cm of xi2_6,align=center] {$i_6$};
\node [below = 0.1cm of xi2_7,align=center] {$i_7$};
\node [below = 0.1cm of xi2_8,align=center] {$i_8$};

\path (xi1_1)edge[->,dashed](u_1_1);
\path (xi1_2)edge[->,dashed](u_1_2);
\path (xi1_3)edge[->,dashed](u_1_3);
\path (xi1_4)edge[->,dashed](u_1_4);
\end{tikzpicture}%
}
\end{center}
\caption{Scenario tree for $\boldsymbol{\xi}$ of Example \ref{onlineQ_eg} and its optimal controls at different times and under different scenarios.}
\label{Q_eg_tree}
\end{figure}
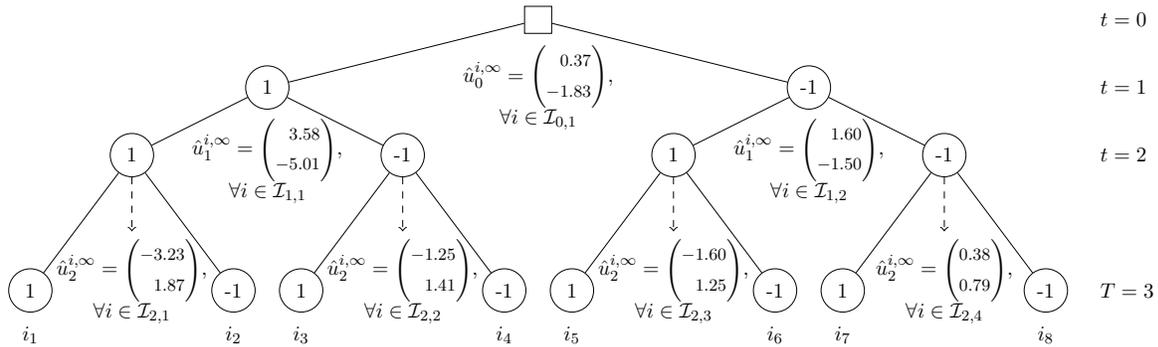

\end{exam}

\section{Dynamic Portfolio Selection with Smoothing Properties}

In this section, let us consider a financial market consisting of $n$ risky assets and one risk-free asset, and an investment time horizon $T$ (with the time indices $t=0,1,\cdots,T-1$). The total return of the riskless asset, denoted by $r_t$, is deterministic and given, whereas the random total return of risky assets at time $t$, denoted by ${e}_t=(e_t^1,\ldots,e_t^n)'\in\mathbb{R}^n$, is assumed to follow a discrete distribution $D_{e_t}$ with $|D_{e_t}|$ possible realizations and corresponding probabilities $\pi_t^k\geq 0$, $k=1,\ldots,|D_{e_t}|$. Furthermore, ${e}_t$'s from different time stages are assumed to be independent.
A series of realizations on $\{{e}_t\}_t$ then defines a \emph{scenario}. Therefore, given the time horizon $T$, there are in total $\prod_{t=0}^{T-1}|D_{e_t}|$ number of scenarios and the scenario probability $\rho_i$ is then calculated by the product of related $\pi_t^k$'s that are attached to this scenario $i$. Let $x_t\in\mathbb{R}$ be the wealth level at time $t$ with the initial wealth $x_0$ given, and ${u}_t=(u_t^1,,\cdots,u_t^n)'\in\mathbb{R}^n$ be the portfolio allocation where $u_t^i$ is the dollar amount to invest in the risky asset $i$, $i=1,\cdots,n$. Then the dollar amount to the riskless asset at time $t$ is $(x_t-\sum_{i=1}^n u_t^i)$ under the assumption of self-financing. Therefore, the wealth dynamic under policy ${u}_t$ becomes
\begin{align}
  x_{t+1} = \sum_{i=1}^n e_t^i {u}_t^i+(x_t-\sum_{i=1}^n {u}_t^i)r_t
          = r_{t}x_t + {P}_t'{u}_t,~t=0,1,\ldots,T-1,\label{wealth_dynamic}
\end{align}
where ${P}_t={e}_t-r_t{1}_n\in\mathbb{R}^n$ is known as the excess total return and ${1}_n\in\mathbb{R}^n$ is an all-one vector of size $n$.

There are in general two directions on objectives for modelling the portfolio selection problem, i.e., the expected utility maximization framework and the mean-variance formulation. Among  conventional formulations, most objective functions  focus on the performance of the \emph{terminal} wealth. Failing to take into account the investment behavior during the investment process could lead to large fluctuations either in the wealth level or in the policy values, while the former may further lead to a bankruptcy (see \cite{Discrete_Bankruptcy} and \cite{Continuous_Bankruptcy}) and the latter may cause large transaction costs. Thus, a relatively smooth wealth growth may often be desirable, even with some sacrifice of the terminal wealth. In some other situations, to avoid the transaction cost as much as possible, investors may demand relatively uniform budget allocation during the whole investment period. In order to reflect these practical concerns, we extend in this research both the traditional utility formulation and the conventional dynamic mean-variance model by attaching to their original objective functions an expectation of a \emph{smoothing} term in a \emph{quadratic variation} form along the time horizon,
\begin{align}
S(\{x_t\}_t,\{u_t\}_t)=\sum_{t\in\mathcal{T}}\Big(f_t(x_t,u_t)
-\frac{1}{|\mathcal{T}|}\sum_{\tau\in\mathcal{T}}f_\tau(x_\tau,u_\tau)\Big)^2,
\label{smoothing_term}
\end{align}
for some types of functions $f_t:\mathbb{R}\times\mathbb{R}^n\rightarrow\mathbb{R}$, where $\mathcal{T}\subseteq \{0,1,\ldots,T\}$ is a subset of time stages selected for smoothing purpose. Some concrete choices of $f_t$ could be
\begin{align}
  f_t(x_t,u_t)=x_t
  \label{wealth_smooth}
\end{align}
in order for us to smooth the wealth levels, and
\begin{align}
  f_t(x_t,u_t)=\sum_{i\in\mathcal{N}}{u}_t^i,
  \label{control_smooth}
\end{align}
in order for us to smooth the total investment amount in some specific risky assets specified by $\mathcal{N}\subseteq\{1,2,\ldots,n\}$.

\subsection{Smoothing under Expected Utility Maximization}

Conventionally, the investor seeks to find the optimal ${u}_t$ for all $t$ such that the expected utility of the final wealth, denoted by $\mathbb{E}[U(x_T)]$, is maximized subject to the wealth dynamic (\ref{wealth_dynamic}), where $U(x)$ is the investor's utility function. In this paper we further assume that the utility function $U(x)$ satisfies $-U'(x)/U''(x)=a+bx$ for certain coefficients $a$ and $b$, which is known as the hyperbolic-absolute-risk-aversion (HARA) utility. Some commonly-used utilities, for example, the exponential utility of the form $\{U(x)=-e^{-x/a}:x\in\mathbb{R}\}$ where $a>0$ and $b=0$ and the power utility of the form $\{U(x)=\frac{1}{b-1}(a+bx)^{1-1/b}:x \geq -a/b\}$ where $b\neq 1$ and $b\neq 0$ are two special cases of HARA utility. In this subsection, we consider expected utility maximization with a general smoothing term,
\begin{IEEEeqnarray*}{ccl}
  ~(\mathcal{US}(\gamma))~~ & \max_{{u}_t,\forall t} & ~\mathbb{E}[U(x_T)]-\gamma \mathbb{E}[S(\{x_t\}_{t\in\mathcal{T}},\{u_t\}_{t\in\mathcal{T}})]\\
  & \rm{s.t.} & ~x_{t+1} = r_{t}x_t + {P}_t'{u}_t,~t=0,1,\ldots,T-1,
\end{IEEEeqnarray*}
where the trade-off parameter $\gamma\geq 0$ specified by the investor represents a trade-off between the expected utility of the \emph{terminal} wealth and the smoothing demand during the \emph{intermediate} process. The larger the $\gamma$, the more the investor is concerned about smoothing. When $\gamma=0$, this problem reduces to the classical model under the HARA utility. It is well known that $(\mathcal{US}(0))$ is solvable by DP and the optimal policy (see \cite{bertsekas2017dynamic}) is given by
\begin{align}
{u}^*_{t}(x_{t})={\beta}_{t}\left(\frac{a}{\prod_{\tau=t+1}^{T-1}r_{\tau}}+br_{t}x_{t}\right),
\label{utility_analytical}
\end{align}
for $t=0,1,\ldots,T-1$ (let us define the operator $\prod_{\tau=T}^{T-1}r_{\tau}=1$ for consistency), where ${\beta}_{t}=(\beta_{t}^1,\ldots,\beta_{t}^n)'\in\mathbb{R}^n$ should be derived from the optimality condition once given $x_t$,
\begin{align}
\mathbb{E}\left[U'\left(r_{t}x_{t}
+\left(\displaystyle\frac{a}{\prod_{\tau=t+1}^{T-1}r_{\tau}}+br_{t}x_{t}\right)
{\beta}'_{t}{P}_{t}\right){P}_{t}\right]={0},
\label{beta_condition}
\end{align}
which is a system of $n$ nonlinear equations at time $t$. In general, the condition (\ref{beta_condition}) is hard to solve for ${\beta}_{t}$. We point out here that $(\mathcal{US}(0))$ is also solvable by PHA under a discrete market setting, leading to a family of numerical optimal solutions in a tabular form, conditional on the future realizations of $e_t$'s. More importantly, PHA will display its extra solution power more on $(\mathcal{US}(\gamma))$ with $\gamma>0$, the expected utility maximization with smoothing term, whose non-separability prevents DP from its adoption. It is obvious that, for any $\gamma\geq 0$, the quadratic smoothing terms in the forms of (\ref{wealth_smooth}) or (\ref{control_smooth}), together with the linear wealth dynamic in (\ref{wealth_dynamic}), make the problem $(\mathcal{US}(\gamma))$  satisfying the conditions in Theorem \ref{PHA_thm}, thus being solvable by PHA.
We complete this subsection by investigating a case study below.

\begin{exam} 
Consider a similar market setting as in Example 3 of \cite{cui2014optimal}, where there are three risky assets ($n=3$) and the distribution, for simplicity, is directly imposed on the random excess total return ${P}_t$, instead of ${e}_t$. Suppose that ${P}_t$ is independent and identically distributed with a discrete uniform distribution of five possible realizations ($|D_{P_t}|=5$ for all $t$ and $\pi_t^k=1/5$ for all $k=1,\ldots,5$ and all $t$),
\begin{align}
{P}_t\in&
  \left(
  \left[
  \begin{array}{c}
  0.18\\
  -0.05\\
  -0.14
  \end{array}
  \right],
  \left[
  \begin{array}{c}
  0.03\\
  -0.12\\
  -0.03
  \end{array}
  \right],
  \left[
  \begin{array}{c}
  -0.05\\
  0.15\\
  0.05
  \end{array}
  \right],
  \left[
  \begin{array}{c}
  -0.01\\
  0.15\\
  0.10
  \end{array}
  \right],
  \left[
  \begin{array}{c}
  -0.05\\
  0.01\\
  0.06
  \end{array}
  \right]
  \right),~\forall t.
  \label{P_distribution}
\end{align}
We scale the initial wealth to $x_0=1$, and set $T=3$ and $r_t=1.04$ for all $t$. Suppose that the investor has an exponential utility $U(x)=-e^{-x}$ (hence $a=1$ and $b=0$). Originally, the model $(\mathcal{US}(0))$ with this exponential utility can still be solved by  DP under the above discrete market setting \eqref{P_distribution}, and based on (\ref{utility_analytical}), the analytical optimal feedback policy is given by, for $t=0,1,\ldots,T-1$,
\begin{eqnarray}
{u}^*_{t}(x_t)=\frac{{\beta}_{t}}{\prod_{\tau=t+1}^{T-1}r_{\tau}},
\label{utility_analytical_concrete}
\end{eqnarray}
According to (\ref{beta_condition}), the $x_t$-dependent ${\beta}_{t}=(\beta_{t}^1,\beta_{t}^2,\beta_{t}^3)'$ at each time $t$ should be derived from the following system of nonlinear equations, starting from $t=0$,
\begin{eqnarray}
  \sum_{k=1}^{|D_{P_t}|}\pi_k\exp\left\{-\left(
  r_tx_t+\frac{\beta'_{t}{P}_{t,k}}{\prod_{\tau=t+1}^{T-1}r_{\tau}}
  \right)\right\}{P}_{t,k}={0},
  \label{beta_condition_concrete}
\end{eqnarray}
where ${P}_{t,k}$ stands for the $k$th possible realization in $D_{P_t}$ of (\ref{P_distribution}). Although obtaining the value of ${\beta}_{t}$ is indispensable for executing the DP-based optimal policy, solving for ${\beta}_{t}$ from (\ref{beta_condition_concrete}) is not easy, even under the current discrete market setting.

We now resolve the above $(\mathcal{US}(0))$ by the scenario-decomposition method PHA and get the optimal asset allocations $\hat{u}_t^{i,\infty}$ in Table \ref{utility_tabular} (rounding in two decimals), which is a \emph{tabular} form in the sense that it indicates how much to invest at what time, on which asset (the symbols A1, A2 and A3 represent the three risky assets, respectively), and under which scenario (a path of realized $P_{t,k}$'s). Then the wealth trajectory under the optimal policy can be traced for any scenario $i$, and we denote it by $\{x_t^{i,\infty}\}_t$ with $x_0^{i,\infty}=x_0$ given for every $i$. Since the number of all possible realizations of wealth trajectories under the optimal solutions is finite in our discrete market (which in this example is $|\mathcal{I}|=125$), we could easily check the consistency of the optimal solutions between numerical values from PHA and those outputted by the analytical policy from DP when plugging in those possible future wealth levels. This is done by, for each scenario $i$, replacing the left hand side of \eqref{utility_analytical_concrete} with the value of $\hat{u}_t^{i,\infty}$ and deriving \emph{reversely} the scenario-specific ${\beta}_{t,i}$, and we succeed to confirm that the resulted ${\beta}_{t,i}$ satisfies the equality in (\ref{beta_condition_concrete}) where we set $x_t=x_t^{i,\infty}$. We again numerically demonstrate the equivalence between the scenario-decomposition and the time-decomposition approaches, when available, for solving separable, monotone and convex multistage decision-making problems under a finite-scenario setting. More importantly, we attain the exact investment decisions that DP often fails to provide due to the difficulty in finding $\beta_t$ from \eqref{beta_condition}.

When we consider $(\mathcal{US}(\gamma))$ with $\gamma>0$, only PHA works for numerical solutions. In this example, we test $\gamma=1$ and $\gamma=10$ for the wealth smoothing term in (\ref{wealth_smooth}) with $\mathcal{T}=\{1,2,3\}$. The computational tabular results are also listed in Table \ref{utility_tabular}. Comparing them with the results without smoothing, we find that, when $\gamma>0$, the asset allocations in general become moderate. This further leads to smoother wealth trajectories, no matter in a single-scenario level $\{x_t^{i,\infty}\}_t$ (which could be seen from our experiments but we omit the details here), or in an overall level in terms of their expectations and variances as exhibited in Table \ref{utility_appropriate} (rounding in two decimals if needed). From Table \ref{utility_appropriate}, we could see a stabler growth on the expected wealth and a less-fluctuated wealth movement (i.e., lower variances) when the wealth smoothing is considered ($\gamma=1,10$). These naturally cause a decrease on the expected terminal wealth $\mathbb{E}[x_T^{i,\infty}]$ compared with the non-smoothing setting ($\gamma=0$). And the larger the $\gamma$, the more conservative the investment decision, thus the bigger the sacrifice on $\mathbb{E}[x_T^{i,\infty}]$. On the other hand, however, the smoothing helps, to a certain degree, on reducing the possible bankruptcy induced by the relatively aggressive investing style during the investment process. To see this, let us define the bankruptcy rate at time $t$ by (similar to \cite{Discrete_Bankruptcy})
\begin{align}
BR_t
&= \mathbb{P}(x_t<x_t^b,~x_\tau\geq x_\tau^b~\text{for }\tau=0,1,\ldots,t-1)\nonumber\\
&=\frac{BN_t}{|\mathcal{I}|-\sum_{\tau=0}^{t-1}BN_{\tau}},~t=1,\ldots,T,
\label{bankruptcy_definition}
\end{align}
where $x_t^b$ denotes the wealth benchmark at time $t$ specified by the investor, and we define $BN_t$ as the number of bankruptcy scenarios at time $t$ under which $x_t<x_t^b$ and $x_\tau\geq x_\tau^b$ for $\tau<t$. In fact, the denominator of (\ref{bankruptcy_definition}) indicates the number of scenarios that still survive at time $t$. Initially at $t=0$, we set $BR_0=0$, $BN_0=0$, and $x_0^b=x_0$, and we choose a risk-free-growing wealth benchmark, that is, $x_t^b=\prod_{\tau=0}^{t-1}r_tx_0$ for $t\geq 1$. From Table \ref{utility_appropriate}, we could see a distinct reduction on the bankruptcy rates after introducing the smoothing property with some appropriate smoothing balances (such as $\gamma=1$ and $\gamma=10$ here). Moreover, adding a smoothing term also leads to a better worst case of the final wealth (in this example we obtain $0.3782$, $0.9853$, and $1.0131$ for $\gamma=0,1,10$, respectively): the conservative behavior under smoothing helps to avoid severe losses in case an adverse scenario occurs.

\begin{table}[t]
\centering
\caption{Tabular Optimal Solutions $\hat{u}_t^{i,\infty}$ of ($\mathcal{US}(\gamma))$ in Example \ref{utility_eg}} 
\resizebox{\columnwidth}{!}{%
\begin{tabular}{ccccccccccc}
\toprule
\multirow{2}[2]{*}{$t$} & \multicolumn{3}{c}{$\gamma=0$} & \multicolumn{3}{c}{$\gamma=1$} & \multicolumn{3}{c}{$\gamma=10$} & \multirow{2}[3]{*}{Scenarios} \\
\cmidrule{2-10} & A1 & A2 & A3 & A1 & A2 & A3 & A1 & A2 & A3 &  \\
\midrule
0  & 10.92 & 3.31 & 7.25 & 13.76 & 0.92 & 14.04 & 12.16 & 0.14 & 13.43 &  \\
\midrule
\multirow{5}[2]{*}{1} & 8.65 & 3.33 & 4.97 & 2.79 & -0.85 & 3.52 & -0.55 & -0.09 & -0.44 & if ${P}_{0,1}$ occurs \\
  & 13.65 & 2.85 & 10.46 & 1.92 & 0.10 & 1.74 & -0.36 & -0.02 & -0.33 & if ${P}_{0,2}$ occurs \\
  & 9.86 & 3.25 & 6.23 & 0.88 & 0.20 & 0.64 & -0.65 & 0.04 & -0.67 & if ${P}_{0,3}$ occurs \\
  & 7.65 & 3.31 & 4.09 & -0.74 & -0.16 & -0.55 & -1.75 & 0.06 & -1.75 & if ${P}_{0,4}$ occurs \\
  & 12.47 & 2.98 & 9.12 & 0.88 & 0.18 & 0.67 & -0.75 & 0.04 & -0.76 & if ${P}_{0,5}$ occurs \\
\midrule
\multirow{25}[10]{*}{2} & 7.12 & 3.31 & 3.56 & 1.78 & -0.60 & 2.31 & -0.49 & -0.05 & -0.42 & if ${P}_{0,1},{P}_{1,1}$ occurs \\
  & 10.64 & 3.06 & 7.14 & 3.57 & -1.37 & 4.81 & -0.74 & -0.07 & -0.64 & if ${P}_{0,1},{P}_{1,2}$ occurs \\
  & 7.81 & 3.31 & 4.16 & 3.25 & -0.66 & 3.78 & -0.51 & -0.12 & -0.37 & if ${P}_{0,1},{P}_{1,3}$ occurs \\
  & 6.48 & 3.26 & 3.06 & 0.72 & -0.69 & 1.39 & -0.50 & 0.00 & -0.48 & if ${P}_{0,1},{P}_{1,4}$ occurs \\
  & 9.68 & 3.18 & 6.07 & 3.31 & -1.22 & 4.40 & -0.68 & -0.07 & -0.58 & if ${P}_{0,1},{P}_{1,5}$ occurs \\
\cmidrule{2-11}  & 10.41 & 3.09 & 6.89 & 1.40 & 0.06 & 1.28 & -0.30 & -0.01 & -0.28 & if ${P}_{0,2},{P}_{1,1}$ occurs \\
  & 17.32 & 2.04 & 15.15 & 2.30 & 0.08 & 2.14 & -0.44 & -0.01 & -0.41 & if ${P}_{0,2},{P}_{1,2}$ occurs \\
  & 12.56 & 2.76 & 9.41 & 2.05 & 0.12 & 1.85 & -0.39 & -0.03 & -0.35 & if ${P}_{0,2},{P}_{1,3}$ occurs \\
  & 9.02 & 3.23 & 5.39 & 1.09 & -0.06 & 1.11 & -0.23 & 0.00 & -0.22 & if ${P}_{0,2},{P}_{1,4}$ occurs \\
  & 15.66 & 2.28 & 13.14 & 2.11 & 0.10 & 1.93 & -0.42 & -0.01 & -0.39 & if ${P}_{0,2},{P}_{1,5}$ occurs \\
\cmidrule{2-11} & 7.96 & 3.31 & 4.31 & 0.64 & 0.16 & 0.46 & -0.53 & 0.04 & -0.54 & if ${P}_{0,3},{P}_{1,1}$ occurs \\
  & 12.25 & 2.81 & 9.03 & 1.04 & 0.25 & 0.75 & -0.72 & 0.05 & -0.75 & if ${P}_{0,3},{P}_{1,2}$ occurs \\
  & 8.92 & 3.25 & 5.28 & 0.91 & 0.17 & 0.71 & -0.80 & 0.05 & -0.82 & if ${P}_{0,3},{P}_{1,3}$ occurs \\
  & 7.13 & 3.29 & 3.61 & 0.60 & 0.09 & 0.48 & -0.27 & 0.02 & -0.29 & if ${P}_{0,3},{P}_{1,4}$ occurs \\
  & 11.17 & 2.97 & 7.78 & 0.93 & 0.25 & 0.65 & -0.70 & 0.05 & -0.72 & if ${P}_{0,3},{P}_{1,5}$ occurs \\
\cmidrule{2-11} & 6.46 & 3.26 & 3.05 & -0.63 & -0.11 & -0.50 & -1.41 & 0.05 & -1.41 & if ${P}_{0,4},{P}_{1,1}$ occurs \\
  & 9.34 & 3.21 & 5.72 & -0.93 & -0.18 & -0.72 & -1.97 & 0.08 & -1.98 & if ${P}_{0,4},{P}_{1,2}$ occurs \\
  & 6.84 & 3.27 & 3.37 & -0.64 & -0.20 & -0.42 & -2.07 & 0.07 & -2.06 & if ${P}_{0,4},{P}_{1,3}$ occurs \\
  & 5.80 & 3.15 & 2.69 & -0.47 & -0.10 & -0.35 & -0.83 & 0.04 & -0.84 & if ${P}_{0,4},{P}_{1,4}$ occurs \\
  & 8.57 & 3.25 & 4.96 & -0.88 & -0.16 & -0.69 & -1.89 & 0.07 & -1.89 & if ${P}_{0,4},{P}_{1,5}$ occurs \\
\cmidrule{2-11} & 9.70 & 3.18 & 6.10 & 0.66 & 0.13 & 0.50 & -0.61 & 0.03 & -0.61 & if ${P}_{0,5},{P}_{1,1}$ occurs \\
  & 15.66 & 2.28 & 13.14 & 1.00 & 0.23 & 0.73 & -0.85 & 0.04 & -0.86 & if ${P}_{0,5},{P}_{1,2}$ occurs \\
  & 11.34 & 2.94 & 7.97 & 0.91 & 0.16 & 0.72 & -0.90 & 0.04 & -0.90 & if ${P}_{0,5},{P}_{1,3}$ occurs \\
  & 8.45 & 3.25 & 4.85 & 0.61 & 0.07 & 0.52 & -0.34 & 0.02 & -0.35 & if ${P}_{0,5},{P}_{1,4}$ occurs \\
  & 14.41 & 2.46 & 11.64 & 0.95 & 0.21 & 0.71 & -0.81 & 0.04 & -0.82 & if ${P}_{0,5},{P}_{1,5}$ occurs \\
\bottomrule
\end{tabular}%
}
\label{utility_tabular}%
\end{table}%

   \begin{table}[t]
     \centering
     \caption{Wealth Statistics and Bankruptcy Evaluations of Example \ref{utility_eg}}
     \resizebox{\columnwidth}{!}{%
       \begin{tabular}{ccccccccccc}
       \toprule
       \multirow{2}[2]{*}{$t$} & \multirow{2}[2]{*}{$x_t^b$} & \multicolumn{3}{c}{$\mathbb{E}[x_t^{i,\infty}]$} & \multicolumn{3}{c}{\text{Var}$(x_t^{i,\infty})$} &  \multicolumn{3}{c}{$BR_t$} \\
\cmidrule{3-11}  &    & $\gamma=0$ & $\gamma=1$ & $\gamma=10$ & $\gamma=0$ & $\gamma=1$ & $\gamma=10$ & $\gamma=0$ & $\gamma=1$ & $\gamma=10$\\
       \midrule
       1  & 1.04 & 1.41 & 1.45 & 1.39 & 0.27 & 0.28 & 0.21 & 0.4 & 0.2 & 0.2\\
       2  & 1.08 & 1.82 & 1.54 & 1.43 & 0.49 & 0.28 & 0.22 & 0 & 0 & 0\\
       3  & 1.13 & 2.23 & 1.63 & 1.46 & 0.66 & 0.28 & 0.22 & 0.03 & 0 & 0\\
       \bottomrule
       \end{tabular}%
       }
     \label{utility_appropriate}%
   \end{table}%
\label{utility_eg}
\end{exam}

\subsection{Smoothing under Mean-variance Formulation}\label{MV_subsection}

Let us now consider a conventional discrete-time mean-variance (MV) formulation given as follows,
\begin{IEEEeqnarray*}{ccl}
~(\mathcal{MV}(w))~~ & \max\limits_{u_t,\forall t} & ~\mathbb{E}(x_T) - w Var(x_T) \\
& {\rm s.t.} &~x_{t+1}=r_tx_t+{P}_t'{u}_t,~t= 0, 1, \ldots, T-1,
\end{IEEEeqnarray*}
where the parameter $w$, predetermined by the investor, explicitly reveals her trade-off between the expected terminal wealth and its variance. Note that $(\mathcal{MV}(w))$, and also other types of MV models, is nonseparable owing to the variance operator. From \cite{li2000optimal}, we know that $(\mathcal{MV}(w))$ can be embedded into a family of separable auxiliary problems that are solvable by DP and the solution of an auxiliary problem with a special value of the parameter in turn solves the primal problem. We list the analytical optimal feedback policy of $(\mathcal{MV}(w))$ below,
{\small
\begin{align}
u_{t}^*(x_{t};w)
={}&-K_tr_tx_{t}+\left(x_0\prod_{s=0}^{T-1}r_s+\frac{1}{2w\prod_{s=0}^{T-1}
(1-\mathbb{E}'[P_{s}]K_{s})}\right)
\left(\prod_{\tau=t+1}^{T-1}\frac{1}{r_\tau}\right)K_t,~t=0,1,\ldots,T-1,
\label{MV_policy}
\end{align}
}%
where $K_t=\mathbb{E}^{-1}\left[P_{t}P_{t}'\right]\mathbb{E}[P_{t}]$, and we define the operator $\prod_{\tau=T}^{T-1}(1/r_\tau)=1$ for consistency.

Owing to the similar issues as in the utility framework, the objective in $(\mathcal{MV}(w))$ merely considers the \emph{final} wealth, thus may suffer possible large fluctuations during the investment process. Therefore, in this subsection we also investigate a more general mean-variance formulation by adding a wealth smoothing term,
\begin{IEEEeqnarray*}{ccl}
~(\mathcal{MVS}(w,\gamma))~~ & \max\limits_{u_t,\forall t} &~ \mathbb{E}(x_T) - w Var(x_T)
-\gamma \mathbb{E}\left[\sum\nolimits_{t=1}^{T}(x_t-\bar x)^2\right] \\
& {\rm s.t.} &~ x_{t+1}=r_tx_t+{P}_t'{u}_t,~t= 0, 1, \ldots, T-1,
\end{IEEEeqnarray*}
where $\bar x =1/T\sum_{t=1}^{T}x_t$ denotes the average wealth \emph{along the time horizon}, and $\gamma\geq 0$ reflects a preselected trade-off between the MV objective and the smoothing term. Note that PHA cannot be directly applied to $(\mathcal{MVS}(w,\gamma))$ since it is originally designed for solving the stochastic problem with only the \emph{risk-neutral} evaluation criterion (i.e., the expectation measure). We first rearrange the objective in $(\mathcal{MVS}(w,\gamma))$ as
{\small
\begin{align}
{}&\mathbb{E}(x_T) - w Var(x_T)-\gamma \mathbb{E}\left[\sum_{t=1}^{T}(x_t-\bar x)^2\right] \nonumber\\
={}&\mathbb{E}[x_T] - w (\mathbb{E}[x_T^2]-\mathbb{E}^2[x_T]) -\gamma \mathbb{E}\left[\sum_{t=1}^{T}x_t^2-\frac{1}{T}(\sum_{t=1}^{T}x_t)^2\right]\nonumber\\
={}&-\mathbb{E}\left[(\gamma-\frac{\gamma}{T}) \sum_{t=1}^{T-1}x_t^2
+(w+\gamma-\frac{\gamma}{T})x_T^2-\frac{\gamma}{T}\sum_{1\leq i\neq j\leq T}x_ix_j\right]
+w\mathbb{E}^2[x_T]+\mathbb{E}[x_T]\nonumber\\
={}&\tilde{U}(\mathbb{E}[x_1^2],\mathbb{E}[x_1x_2],\ldots,\mathbb{E}[x_1x_T],\mathbb{E}[x_2x_1],
\mathbb{E}[x_2^2],\ldots,\mathbb{E}[x_2x_T],
\ldots,\mathbb{E}[x_Tx_{1}],
\mathbb{E}[x_Tx_{2}],\ldots,\mathbb{E}[x_T^2],\mathbb{E}[x_T]).
\label{MVS_objective}
\end{align}
}%
Note that $\tilde{U}$ is a convex function of
$\mathbb{E}[x_ix_j],~1\leq i,j\leq T$, and $\mathbb{E}[x_T]$. By invoking the embedding scheme as in \cite{li2000optimal}, we consider the following auxiliary problem
\begin{IEEEeqnarray*}{ccl}
~(\mathcal{A}(w^c,\lambda))~~& \max\limits_{u_t,\forall t} & ~\mathbb{E}\left[-\sum_{1\leq i,j\leq T}w_{ij}x_ix_j+\lambda x_T\right] \\
& {\rm s.t.} & ~x_{t+1}=r_tx_t+{P}_t'{u}_t,~t= 0, 1, \ldots, T-1,
\end{IEEEeqnarray*}
where $\lambda\in\mathbb{R}$ and we assemble all the $w_{ij}$'s into
\begin{align*}
  w^c
  ={}&(w_{11},w_{12},\ldots,w_{1T},w_{21},w_{22},\ldots,
  w_{2T},\ldots,
  w_{T1},w_{T2},\ldots,w_{TT})'\in\mathbb{R}^{T^2},
\end{align*}
with $w_{tt}=\gamma-\gamma/T$ for $t=1,\ldots,T-1$, $w_{TT}=w+\gamma-\gamma/T$ and $w_{ij}=-\gamma/T$ for $1\leq i,j\leq T$ with $i\neq j$. Let us further denote the solution set of $(\mathcal{MVS}(w,\gamma))$ by $\Pi(w,\gamma)$, and the solution set of $(\mathcal{A}(w^c,\lambda))$ by $\Pi_\mathcal{A}(w^c,\lambda)$. We also denote, for any policy $\mathbf{u}=(u_0',u_1',\ldots,u_{T-1}')'\in\mathbb{R}^{nT}$, the first-order derivative of $\tilde{U}$ w.r.t. $\mathbb{E}[x_T]$ by
\begin{eqnarray}
  d(\mathbf{u})=\frac{d\tilde{U}}{d\mathbb{E}[x_T]}\bigg|_\mathbf{u}=1+2w\mathbb{E}[x_T]|_\mathbf{u}.
  \label{U_to_EX}
\end{eqnarray}
\begin{lem}
  For any $\mathbf{u}^*\in \Pi(w,\gamma)$, $\mathbf{u}^*\in \Pi_\mathcal{A}(w^c,d(\mathbf{u}^*))$.
  \label{MV_thm1}
\end{lem}
\begin{proof}
As $\tilde{U}$ is convex w.r.t. $\mathbb{E}[x_ix_j]$, $1\leq i,j\leq T$, and $\mathbb{E}[x_T]$, the proof is similar to Theorem 1 in \cite{li2000optimal}. Thus, we omit the details here.
\end{proof}

The interpretation of Lemma \ref{MV_thm1} is similar to \cite{li2000optimal}, that is, in order to obtain the primal solution, the problem $(\mathcal{MVS}(w,\gamma))$ can be embedded into the auxiliary problem $(\mathcal{A}(w^c,\lambda))$. Moreover, the auxiliary problem in \cite{li2000optimal} is a special case of ours: only $\mathbb{E}[x^2_{T}]$ exists in the auxiliary problem of \cite{li2000optimal}, while all cross terms, $\mathbb{E}[x_ix_j],~1\leq i,j\leq T$, appear in our setting. What significantly distinguishes our auxiliary problem  from those in \cite{li2000optimal} and \cite{Discrete_Bankruptcy} is that $(\mathcal{A}(w^c,\lambda))$ in our case cannot be solved by DP anymore, since the smoothing introduces cross terms of wealth levels among different time stages. Before we demonstrate that $(\mathcal{A}(w^c,\lambda))$ can be solved by PHA, we need to prove first that $(\mathcal{A}(w^c,\lambda))$ satisfies the conditions in Theorem \ref{PHA_thm}. To see this, let us rewrite the auxiliary problem $(\mathcal{A}(w^c,\lambda))$ as the following equivalent compact form,
\begin{IEEEeqnarray*}{ccl}
~(\mathcal{A}(\mathbf{W},\lambda))~~ & \max\limits_{\mathbf{u}} & ~\mathbb{E}\left[-\mathbf{x}'\mathbf{W}\mathbf{x}+\lambda \mathbf{x}'\delta\right] \\
&{\rm s.t.} & ~\mathbf{x}=\mathbf{P}\mathbf{u}+x_0\mathbf{r},
\end{IEEEeqnarray*}
where $\mathbf{x}=(x_1,\ldots,x_T)'\in\mathbb{R}^T$, $\delta=(0,\ldots,0,1)'\in\mathbb{R}^T$, and
\begin{align*}
  \mathbf{W}
  &=(w_{ij})_{1\leq i,j\leq T}
  =\left(
  \begin{array}{cccc}
    \gamma-\frac{\gamma}{T} & -\frac{\gamma}{T} & \ldots & -\frac{\gamma}{T}\\
    -\frac{\gamma}{T} & \gamma-\frac{\gamma}{T} & \ldots & -\frac{\gamma}{T}\\
    \vdots & \vdots & \ddots & \vdots \\
    -\frac{\gamma}{T} & -\frac{\gamma}{T} & \ldots & w+\gamma-\frac{\gamma}{T}
  \end{array}
  \right)\in\mathbb{R}^{T\times T},\\
  \mathbf{P}&=\left(
  \begin{array}{cccc}
    {P}_0' & {0} & \ldots & {0}\\
    r_1{P}_0' & {P}_1' & \ldots & {0}\\
    r_2r_1{P}_0' & r_2{P}_1' & \ldots & {0}\\
    \vdots & \vdots  & \ddots & \vdots \\
    (\prod\limits_{t=1}^{T-1}r_t){P}_0' & (\prod\limits_{t=2}^{T-1}r_t){P}_1' &
    \ldots & {P}_{T-1}'
  \end{array}
  \right)\in\mathbb{R}^{T\times nT},\\
  \mathbf{r}&=\left(r_0,r_1r_0,\ldots,\prod\nolimits_{t=0}^{T-1}r_t\right)'\in\mathbb{R}^{T}.
\end{align*}
Given $w\geq 0$ and $\gamma\geq 0$, we have $\mathbf{x}'\mathbf{W}\mathbf{x}=w\mathbb{E}[x_T^2]+\gamma \mathbb{E}\left[\sum_{t=1}^{T}(x_t-\bar x)^2\right]\geq 0$ for any $\mathbf{x}$, thus the matrix $\mathbf{W}$ is positive semidefinite. Together with the fact that $\mathbf{x}$ is linear in $\mathbf{u}$, we conclude that each scenario subproblem of $(\mathcal{A}(\mathbf{W},\lambda))$ (with a certain realization on the matrix $\mathbf{P}$) is concave w.r.t. $\mathbf{u}$. Thus, $(\mathcal{A}(\mathbf{W},\lambda))$ satisfies the conditions in Theorem \ref{PHA_thm}. Notice that the structure of $(\mathcal{A}(\mathbf{W},\lambda))$ falls into the framework of the online QP discussed in Section \ref{QP_section}, except that the underlying systems dynamics are slightly different. Before presenting the solution algorithm, let us consider the condition under which the solution of $(\mathcal{A}(\mathbf{W},\lambda))$ also constitutes a solution to $(\mathcal{MVS}(w,\gamma))$.

\begin{thm}
  Suppose $\mathbf{u}^*\in \Pi_\mathcal{A}(w^c,\lambda^*)$. A necessary condition for $\mathbf{u}^*\in \Pi(w,\gamma)$ is $\lambda^*=1+2w\mathbb{E}[x_T]|_{\mathbf{u}^*}$.
  \label{MV_thm2}
\end{thm}
\begin{proof}
The solution set $\Pi_\mathcal{A}(w^c,\lambda)$ can be characterized by $\lambda$ when we fix $w^c$. Note that from Lemma \ref{MV_thm1} we have $\Pi(w,\gamma)\subseteq\cup_\lambda \Pi_\mathcal{A}(w^c,\lambda)$. Therefore, solving $(\mathcal{MVS}(w,\gamma))$ is equivalent to considering the following,
\begin{align*}
  &\max_\lambda~\tilde{U}(\mathbb{E}[x_i(w^c,\lambda)x_j(w^c,\lambda)],\forall i,j,\mathbb{E}[x_T(w^c,\lambda)])\nonumber\\
  ={}&\max_\lambda~-\left(\sum\nolimits_{1\leq i,j\leq T}w_{ij}\mathbb{E}[x_i(w^c,\lambda)x_j(w^c,\lambda)]\right)
  +w\mathbb{E}^2[x_T(w^c,\lambda)]
  +\mathbb{E}[x_T(w^c,\lambda)].
\end{align*}
The first-order necessary optimality condition for $\lambda^*$ is
\begin{align}
  &-\left(\sum\nolimits_{1\leq i,j\leq T}w_{ij}\frac{d\mathbb{E}[x_i(w^c,\lambda)x_j(w^c,\lambda)]}{d\lambda}\bigg|_{\lambda^*}\right)
  +(1+2w\mathbb{E}[x_T]|_{\mathbf{u}^*})\frac{d\mathbb{E}[x_T(w^c,\lambda)]}{d\lambda}\bigg|_{\lambda^*}=0.
  \label{MV_condition1}
\end{align}
On the other hand, as $\mathbf{u}^*\in \Pi_\mathcal{A}(w^c,\lambda^*)$, we have the following according to \cite{reid1971noninferior},
\begin{align}
  &-\left(\sum\nolimits_{1\leq i,j\leq T}w_{ij}\frac{d\mathbb{E}[x_i(w^c,\lambda)x_j(w^c,\lambda)]}{d\lambda}\bigg|_{\lambda^*}\right)
  +\lambda^*\frac{d\mathbb{E}[x_T(w^c,\lambda)]}{d\lambda}\bigg|_{\lambda^*}=0.
  \label{MV_condition2}
\end{align}
Combining (\ref{MV_condition1}) and (\ref{MV_condition2}), the vector $(-(w^c)',\lambda^*)'$ should be proportional to the vector $(-(w^c)',1+2w\mathbb{E}[x_T]|_{\mathbf{u}^*})'$, thus we must have $\lambda^*=1+2w\mathbb{E}[x_T]|_{\mathbf{u}^*}$.
\end{proof}

We now apply our scenario-decomposition solution method to solve $(\mathcal{A}(\mathbf{W},\lambda))$ for a given $\lambda$. We first deal with individual scenario subproblems in their equivalent minimization forms,
\begin{IEEEeqnarray*}{ccl}
 ~(\mathcal{A}^i(\mathbf{W},\lambda))~~ & \min\limits_{\mathbf{u}} & ~\mathbf{x}'\mathbf{W}\mathbf{x}-\lambda \mathbf{x}'\delta \\
 & {\rm s.t.} & ~\mathbf{x}=\mathbf{P}^i\mathbf{u}+x_0\mathbf{r},
\end{IEEEeqnarray*}
where $\mathbf{P}^i$ is the realization of $\mathbf{P}$ under the scenario $i\in\mathcal{I}$.
Although $(\mathbf{P}^i)'\mathbf{W}\mathbf{P}^i$ could be singular for some $i$, we could always leverage on any convex optimization algorithm to find the global optima of $(\mathcal{A}^i(\mathbf{W},\lambda))$, which are denoted by $\mathbf{u}^{i,0}$. We next consider its augmented Lagrangian at the iteration $\nu$,
\begin{IEEEeqnarray*}{ccl}
 ~(\mathcal{A}^{i,\nu}(\mathbf{W},\lambda))~~ & \min\limits_{\mathbf{u}} & ~\mathbf{x}'\mathbf{W}\mathbf{x}-\lambda \mathbf{x}'\delta + \mathbf{u}'\mathbf{w}^{i,\nu}
 +\frac{1}{2}\alpha|\mathbf{u}-\hat{\mathbf{u}}^{i,\nu}|_2^2 \\
 & {\rm s.t.} & ~\mathbf{x}=\mathbf{P}^i\mathbf{u}+x_0\mathbf{r},
\end{IEEEeqnarray*}
and the optimal solution of $(\mathcal{A}^{i,\nu}(\mathbf{W},\lambda))$, denoted by $\mathbf{u}^{i,\nu+1}$, can always be analytically obtained as
\begin{align}
  \mathbf{u}^{i,\nu+1}={}&-[2(\mathbf{P}^i)'\mathbf{W}\mathbf{P}^i+\alpha I_{nT}]^{-1}
  [2x_0(\mathbf{P}^i)'\mathbf{W}\mathbf{r}-\lambda(\mathbf{P}^i)'\delta +\mathbf{w}^{i,\nu}
  -\alpha\hat{\mathbf{u}}^{i,\nu}],
\end{align}
for some given $\hat{\mathbf{u}}^{i,\nu}$ and $\mathbf{w}^{i,\nu}$. Then the new implementable policy $\hat{\mathbf{u}}^{i,\nu+1}$ is calculated according to (\ref{conditional_expectation}) or following the projection procedure from (\ref{nu+1_sol}) to (\ref{nu+1_project2}). The iteration process continues until the stopping condition in (\ref{Q_stopping criterion}) is satisfied. We finally obtain the optimal solution of $(\mathcal{A}(\mathbf{W},\lambda))$ for a certain $\lambda\in\mathbb{R}$. Now we need to design a solution method to find the optimal $\lambda^*$. When the optimal solution of $(\mathcal{A}(\mathbf{W},\lambda))$ can be expressed in a function form of $\lambda$,  we can substitute it back to $(\mathcal{MVS}(w,\gamma))$ and find the optimal $\lambda^*$ such that $\tilde{U}$ is maximized. In the current situation, however, as PHA does not yield an analytical solution, we need to invoke a heuristic method to carry out the job.

Theorem \ref{MV_thm2} reveals the connection between $\lambda^*$ and $\mathbf{u}^*$ as the optimal solution to \emph{both} $(\mathcal{A}(\mathbf{W},\lambda^*))$ and $(\mathcal{MVS}(w,\gamma))$. In fact, we could rely on this relationship to narrow down the possible range of $\lambda^*$. More precisely, the analysis of the previous subsection
indicates that there is a sacrifice on the expected final wealth when we consider the smoothing. Then according to Theorem $\ref{MV_thm2}$, we could set the upper bound of $\lambda^*$ as
\begin{align}
  \lambda_{\text{max}}=1+2w\mathbb{E}[x_T]|_{\mathbf{u}^{\text{ns}}},
  \label{lambda_max}
\end{align}
where $\mathbf{u}^{\text{ns}}=\{u_t^{\text{ns}}\}_t$ denotes the optimal policy of the classical dynamic MV model with no smoothing term given in (\ref{MV_policy}). On the other hand, we could anticipate that the expected terminal wealth, under an optimal policy in the dynamic MV model with a smoothing term, should be larger than the initial wealth. Thus, we set the lower bound by
\begin{eqnarray}
  \lambda_{\text{min}}=1+2w x_0.
  \label{lambda_min}
\end{eqnarray}
In summary, we claim that $\lambda^*\in[\lambda_{\text{min}},\lambda_{\text{max}}]$. Within this specified range, we could use a line search method to efficiently find the optimal value of $\lambda$ and hence the optimal policy of the primal problem $(\mathcal{MVS}(w,\gamma))$. We summarize our search procedure in details in Algorithm \ref{lambda_algo}. From our extensive experimental studies, we essentially find that the value of $\tilde{U}$ is always concave w.r.t. $\lambda$. This phenomenon was also analytically found in \cite{li2000optimal}. Therefore, we add one more step to fit a quadratic function in the algorithm in order to enhance the accuracy of the ordinary line search. We discuss a case study of a dynamic MV problem with a smoothing term in the following to complete this subsection. 

\begin{algorithm}[!t]
\caption{Find $\lambda^*$ of $(\mathcal{A}(\mathbf{W},\lambda))$ and $\mathbf{u}^*$ of 
$(\mathcal{MVS}(w,\gamma))$}
\label{lambda_algo}
\textbf{Input}: The parameters $w$ and $\gamma$, the distribution $D_{e_t}$ and the risk-free rate for $t=0,1,\ldots,T-1$, the initial wealth $x_0$, the penalty $\alpha$, the tolerance $\epsilon$, and the step size $\theta$.\\
\textbf{Output}: $\lambda^*$ of $(\mathcal{A}(\mathbf{W},\lambda))$ and 
$\mathbf{u}^*$ of $(\mathcal{MVS}(w,\gamma))$.\\
\vspace{-0.4cm}
\begin{algorithmic}[1] 
\STATE Decide $\lambda_{\text{max}}$ by (\ref{lambda_max}) and $\lambda_{\text{min}}$ by (\ref{lambda_min}). Let $\kappa=0$.
\STATE Set $\lambda^\kappa=\lambda_{\text{min}}+\kappa\theta$, and solve $(\mathcal{A}(\mathbf{W},\lambda^\kappa))$, and denote the optimal solution by $\mathbf{u}_\mathcal{A}(\lambda^\kappa)$.
\STATE Compute $\tilde{U}|_{\mathbf{u}_\mathcal{A}(\lambda^\kappa)}$ of (\ref{MVS_objective}) and denote its value by $\tilde{U}(\lambda^\kappa)$.
\STATE If $\lambda^\kappa=\lambda_{\text{max}}$, then stop; else, $\kappa~\leftarrow~\kappa+1$, and go back to Step 2.
\STATE Fit the dataset $\{\lambda^\kappa,\tilde{U}(\lambda^\kappa)\}_{\kappa}$ into a quadratic function (a downward parabola in $\mathbb{R}^2$) and find its optimal solution denoted by $\lambda^{\text{fit}}$.
\STATE Finally, $\lambda^*\in\argmax \{\tilde{U}(\lambda):
\lambda\in\{\lambda^\kappa,\forall\kappa\}\cup\{\lambda^{\text{fit}}\}\}$ and hence $\mathbf{u}^*=\mathbf{u}_\mathcal{A}(\lambda^*)$.
\end{algorithmic}
\end{algorithm}

\begin{exam}\label{MV_eg}
Let us consider a market with the following expectation vector and covariance matrix of the random total return $e_t\in\mathbb{R}^3$, which has been investigated in \cite{li2000optimal},
\begin{align*}
 \mathbb{E}[{e}_t]=\left(
 \begin{array}{c}
 1.162\\
 1.246\\
 1.228
 \end{array}
 \right),~
 cov({e}_t)=\left[
 \begin{array}{ccc}
   0.0146 & 0.0187 & 0.0145 \\
   0.0187 & 0.0854 & 0.0104 \\
   0.0145 & 0.0104 & 0.0289
 \end{array}
 \right].
\end{align*}
We randomly generate discrete distributions in this example to match exactly the above two moments (and, at the same time, prevent arbitrage opportunities as a conventional assumption in the finance literature), so that we could easily verify pros and cons of adding a smoothing term when we compare it with the classical results in \cite{li2000optimal}. More precisely, we begin with an initial wealth $x_0=10$, an investment horizon $T=3$, and one risk-free bond with the total rate $r_t=1.04$. Suppose that $e_t$ follows different uniform distributions at different time $t=0,1,2$ (but independent across time stages), which are given below,
\begin{align*}
e_0\in&
\left\{
\left(
\begin{array}{c}
    1.2722 \\
    1.4294 \\
    1.3126
\end{array}
\right),
\left(
\begin{array}{c}
1.3352 \\
1.4018 \\
1.2519
\end{array}
\right),
\left(
\begin{array}{c}
1.0996  \\
1.3859  \\
1.2868
\end{array}
\right),
\left(
\begin{array}{c}
0.9448    \\
0.6111    \\
0.8722
\end{array}
\right),
\left(
\begin{array}{c}
1.1904   \\
0.8172   \\
1.4877
\end{array}
\right),\right.\\
&\left.\left(
\begin{array}{c}
1.0606   \\
1.1211   \\
1.2403
\end{array}
\right),
\left(
\begin{array}{c}
1.0363   \\
1.4716   \\
1.0339
\end{array}
\right),
\left(
\begin{array}{c}
1.3191    \\
1.4242    \\
1.4224
\end{array}
\right),
\left(
\begin{array}{c}
1.1968   \\
1.5481   \\
1.1376
\end{array}
\right),
\left(
\begin{array}{c}
1.1649\\
1.2495\\
1.2346
\end{array}
\right)
\right\},~|D_{e_0}|=10,
\end{align*}
\begin{align*}
e_1\in&\left\{
\left(
\begin{array}{c}
    1.3056  \\
    1.2997  \\
    1.4462
\end{array}
\right),
\left(
\begin{array}{c}
      1.1498  \\
      1.4673  \\
     1.0048
\end{array}
\right),
\left(
\begin{array}{c}
      1.0833  \\
     0.9035  \\
     1.2252
\end{array}
\right),
\left(
\begin{array}{c}
     0.9665 \\
      0.7577 \\
     0.9926
\end{array}
\right),
\left(
\begin{array}{c}
   1.2876  \\
       1.5520   \\
       1.2225
\end{array}
\right),\right.\\
&\left.
\left(
\begin{array}{c}
  1.2733  \\
     1.1900   \\
     1.4485
\end{array}
\right),
\left(
\begin{array}{c}
      1.0679\\
     1.5517\\
      1.2561
\end{array}
\right)
\right\},~|D_{e_1}|=7,
\end{align*}
\begin{align*}
e_2\in&\left\{
\left(
\begin{array}{c}
    1.0724 \\
    0.7472 \\
    1.1059
\end{array}
\right),
\left(
\begin{array}{c}
       1.0976  \\
      1.0795  \\
       1.3050
\end{array}
\right),
\left(
\begin{array}{c}
      1.3114  \\
    1.5110 \\
       1.2140
\end{array}
\right),
\left(
\begin{array}{c}
     1.3031    \\
      1.4376  \\
      1.5043
\end{array}
\right),
\left(
\begin{array}{c}
    1.0255\\
      1.4547\\
     1.0109
\end{array}
\right)
\right\},~|D_{e_2}|=5.
\end{align*}
Therefore, we have $|\mathcal{I}|=350$ scenarios. The corresponding scenario tree and scenario partitions and bundles can also be easily constructed and obtained so that we omit the details here due to the space limit. We then solve $(\mathcal{MVS}(w,\gamma))$ by the procedure introduced in this subsection for $\gamma=1$ but with different $w=0.5,1,5$, respectively, and obtain the optimal allocation, denoted by $\{\hat{\mathbf{u}}^{i,\infty}(w,\gamma)\}_i$, and calculate the wealth trajectory $\mathbf{x}^{i,\infty}(w,\gamma)$ under $\hat{\mathbf{u}}^{i,\infty}(w,\gamma)$ for all the possible scenarios. We also obtain the optimal policy ${\mathbf{u}}^{*}(x;w)$ of $(\mathcal{MV}(w))$ for the same $w$'s based on (\ref{MV_policy}) and calculate the corresponding wealth trajectories $\mathbf{x}^{i,DP}(w)$ under all circumstances starting from $x_0$.

The statistical results are listed in Table \ref{MV_table} (rounding in four decimals if needed). It is obvious that, in general, taking smoothing into account facilitates investors to better manage their intermediate wealth fluctuations, and this can be seen from the much lower variances under $(\mathcal{MVS}(w,\gamma))$ across all $w$'s considered, compared to those under $(\mathcal{MV}(w))$. Similar to the utility framework, there is also a sacrifice on the expected terminal wealth in $(\mathcal{MVS}(w,\gamma))$ at all levels of $w$. What different from the utility case is that the bankruptcy almost disappears in the current MV example when we set the bankruptcy boundary at $x_t^b=0$ for all $t$ (except for a very small positive bankruptcy rate 0.0143 when $t=2$ under $\mathcal{MV}(0.5)$). It seems plausible that smoothing has little to do on controlling the bankruptcy rate in MV models. However, we claim that smoothing is still a better choice if the investor really cares about the worst case. To see this, suppose the investor does not consider smoothing. Although she could still achieve relatively good management for extreme situations through increasing $w$ (i.e., emphasizing more on the variance part) in $\mathcal{MV}(w)$ (and this is evidenced from the worst-case column from $\mathcal{MV}(0.5)$ to $\mathcal{MV}(5)$ in Table \ref{MV_table}), it costs her nearly a half drop on the expected terminal wealth (from $25.1709$ to $12.6409$ in our experiments). On the other hand, smoothing not only brings better worst cases at every level of $w$, but also makes the investor only suffer a quite mild loss on her expected final wealth (from $13.3638$ to $11.8048$ shown under $\mathcal{MVS}(w,\gamma)$'s).

   \begin{table}[t]
     \centering
     \caption{Statistics of Wealth Levels under $\mathcal{MVS}(w,\gamma)$ and $\mathcal{MV}(w)$ with $w=0.5,1,5$ and $\gamma=1$}
     \resizebox{\columnwidth}{!}{%
       \begin{tabular}{cccccccccc}
       \toprule
       \multirow{2}[2]{*}{$t$} & \multicolumn{4}{c}{$\mathcal{MVS}(0.5,1)$} & & \multicolumn{4}{c}{$\mathcal{MV}(0.5)$} \\
\cmidrule{2-5}\cmidrule{7-10} & $\mathbb{E}[x_t^{i,\infty}]$  & Var$(x_t^{i,\infty})$ & $BR_t$ & Worst case & & $\mathbb{E}[x_t^{i,DP}]$& Var$(x_t^{i,DP})$ & $BR_t$ & Worst case \\
       \midrule
       1  & 12.3502 & 2.8302 & 0  & 7.8774 && 18.5926 & 45.9106 & 0  & 1.0500 \\
       2  & 12.8505 & 2.0145 & 0  & 7.5099 && 22.7971 & 28.3624 & 0.0143 & -6.3719 \\
       3  & 13.3638 & 1.3668 & 0  & 7.4497 && 25.1709 & 13.9223 & 0  & -7.4081 \\
       \midrule
       \midrule
       \multirow{2}[2]{*}{$t$} & \multicolumn{4}{c}{$\mathcal{MVS}(1,1)$} & & \multicolumn{4}{c}{$\mathcal{MV}(1)$} \\
\cmidrule{2-5}\cmidrule{7-10} & $\mathbb{E}[x_t^{i,\infty}]$  & Var$(x_t^{i,\infty})$ & $BR_t$ & Worst case & & $\mathbb{E}[x_t^{i,DP}]$& Var$(x_t^{i,DP})$ & $BR_t$ & Worst case \\
       \midrule
       1  & 11.6889 & 1.2014 & 0  & 8.7909 && 14.4963 & 11.4776 & 0  & 5.7250 \\
       2  & 12.1788 & 0.7319 & 0  & 8.5394 && 16.8066 & 7.0906 & 0  & 2.2220 \\
       3  & 12.6825 & 0.4080 & 0  & 8.6301 && 18.2098 & 3.4806 & 0  & 1.9203 \\
       \midrule
       \midrule
       \multirow{2}[2]{*}{$t$} & \multicolumn{4}{c}{$\mathcal{MVS}(5,1)$} & & \multicolumn{4}{c}{$\mathcal{MV}(5)$} \\
\cmidrule{2-5}\cmidrule{7-10} & $\mathbb{E}[x_t^{i,\infty}]$  & Var$(x_t^{i,\infty})$ & $BR_t$ & Worst case & & $\mathbb{E}[x_t^{i,DP}]$& Var$(x_t^{i,DP})$ & $BR_t$ & Worst case \\
       \midrule
       1  & 10.9083 & 0.1788 & 0  & 9.8317 && 11.2193 & 0.4591 & 0  & 9.4650 \\
       2  & 11.3420 & 0.0823 & 0  & 9.9199 && 12.0141 & 0.2836 & 0  & 9.0972 \\
       3  & 11.8048 & 0.0392 & 0  & 10.3300 && 12.6409 & 0.1392 & 0  & 9.3830 \\
       \bottomrule
       \end{tabular}%
       }%
     \label{MV_table}%
   \end{table}%

\end{exam}

\section{Conclusion}

Stochastic control problems can be in general classified into two categories: separable and nonseparable. The former class can be solved by dynamic programming (DP), at least theoretically. For the latter one, however, no general solution framework has been developed so far in the literature. Recognizing the applicability of progressive hedging algorithm (PHA) in dealing with nonseparable stochastic control problems, we develop in this paper the scenario decomposition solution framework to fill in the gap. To the best of our knowledge, this is the first attempt in the literature  to solve nonseparable stochastic control problems under a general framework. We believe that our new development will greatly extend the reach of the stochastic control. Our results in the online quadratic programming and dynamic portfolio selections with smoothing properties clearly demonstrate the applicabilities of the scenario decomposition approach when the time decomposition methodology inherent in DP fails. We would like to point out one future research direction: While the curse of dimensionality blocks DP from solving relatively large-scale problems, the curse of dimensionality also affects negatively the performance of PHA, especially due to the model assumption of a tree structure.


\newpage
\bibliography{Nonsep_PHA}

\end{document}